\documentclass[12pt]{amsart}
\usepackage[margin=1in]{geometry}

\usepackage{todonotes}

\usepackage{amsthm,amsmath,amssymb}
\usepackage{enumerate}
\usepackage{multicol}
\usepackage{float}
\usepackage{mathtools}

\theoremstyle{plain}
\newtheorem{thm}{Theorem}
\newtheorem{lem}[thm]{Lemma}
\newtheorem{cor}[thm]{Corollary}
\newtheorem{prop}[thm]{Proposition}
\newtheorem{conj}[thm]{Conjecture}
\newtheorem*{mainthm}{Main Theorem}

\theoremstyle{definition}
\newtheorem*{defi}{Definition}
\newtheorem*{ex}{Example}
\newtheorem*{rmk}{Remark}

\newcommand{\cH}{\mathcal{H}}
\newcommand{\cV}{\mathcal{V}}
\newcommand{\cE}{\mathcal{E}}
\newcommand{\cI}{\mathcal{I}}
\newcommand{\cC}{\mathcal{C}}
\newcommand{\cL}{\mathcal{L}}

\newcommand{\cG}{\mathcal{G}}
\newcommand{\cP}{\mathcal{P}}

\newcommand{\N}{\mathbb{N}}
\newcommand{\R}{\mathbb{R}}

\newcommand{\comple}{\preccurlyeq}

\newcommand{\symd}{\mathop{\triangle}}
\newcommand{\wo}{\setminus}
\newcommand{\of}{\subseteq}

\DeclarePairedDelimiter\set{\{}{\}}
\DeclarePairedDelimiterX\setof[2]{\{}{\}}{#1\,:\,#2}
\DeclarePairedDelimiter\abs{\lvert}{\rvert}
\DeclarePairedDelimiter{\floor}{\lfloor}{\rfloor}

\DeclareMathOperator{\ind}{ind}

\usepackage{pgf,tikz}
\usetikzlibrary{arrows,shapes}
\usetikzlibrary{decorations.pathreplacing}
\usepackage{tkz-berge}
\usepackage{ytableau}
\usepackage{mathdots}
\usepackage{verbatim}

\definecolor{light-gray}{gray}{0.95}
\definecolor{medium-gray}{gray}{0.65}
\definecolor{dark-gray}{gray}{0.35}
\newcounter{x}
\newcounter{y}
\newcounter{z}

\newcommand\xaxis{210}
\newcommand\yaxis{-30}
\newcommand\zaxis{90}

\newcommand\topside[3]{
  \fill[fill=light-gray, draw=black,shift={(\xaxis:#1)},shift={(\yaxis:#2)},
  shift={(\zaxis:#3)}] (0,0) -- (30:1) -- (0,1) --(150:1)--(0,0);
}

\newcommand\leftside[3]{
  \fill[fill=medium-gray, draw=black,shift={(\xaxis:#1)},shift={(\yaxis:#2)},
  shift={(\zaxis:#3)}] (0,0) -- (0,-1) -- (210:1) --(150:1)--(0,0);
}

\newcommand\rightside[3]{
  \fill[fill=dark-gray, draw=black,shift={(\xaxis:#1)},shift={(\yaxis:#2)},
  shift={(\zaxis:#3)}] (0,0) -- (30:1) -- (-30:1) --(0,-1)--(0,0);
}

\newcommand\cube[3]{
  \topside{#1}{#2}{#3} \leftside{#1}{#2}{#3} \rightside{#1}{#2}{#3}
}

\newcommand\planepartition[1]{
 \setcounter{x}{-1}
  \foreach \a in {#1} {
    \addtocounter{x}{1}
    \setcounter{y}{-1}
    \foreach \b in \a {
      \addtocounter{y}{1}
      \setcounter{z}{-1}
      \foreach \c in {1,...,\b} {
        \addtocounter{z}{1}
        \cube{\value{x}}{\value{y}}{\value{z}}
      }
    }
  }
}

\title{Maximizing 2-Independent Sets in 3-Uniform Hypergraphs}
\author{L. Keough}
\author{A.J. Radcliffe}

\begin{document}

\maketitle

\begin{abstract}
There has been interest recently in maximizing the number of independent sets in graphs.  For example, the Kahn-Zhao theorem gives an upper bound on the number of independent sets in a $d$-regular graph.  Similarly, it is a corollary of the Kruskal-Katona theorem that the lex graph has the maximum number of independent sets in a graph of fixed size and order.  In this paper we solve two equivalent problems. 

The first is: what $3$-uniform hypergraph on a ground set of size $n$, having at least $t$ edges, has the most $2$-independent sets? Here a $2$--independent set is a subset of vertices containing fewer than $2$ vertices from each edge. This is equivalent to the problem of determining which graph on $n$ vertices having at least $t$ triangles has the most independent sets. The (hypergraph) answer is that, ignoring some transient and some persistent exceptions, a $(2,3,1)$-lex style $3$-graph is optimal.

We also discuss the problem of maximizing the number of $s$-independent sets in $r$-uniform hypergraphs of fixed size and order, proving some simple results, and conjecture an asymptotically correct general solution to the problem.
\end{abstract}

\section{Introduction}

For many years, there has been interest in finding the maximum size of a variety of sub-structures (such as independent sets or matchings) in a graph satisfying certain conditions.  In recent years, there has been increased interest in extremal questions about the \emph{number} of these sub-structures.   That is, rather than asking for the size of the largest independent set, one could ask which graph has the most independent sets, given some set of conditions.
In fact, many extremal problems for the the number of independent sets have been studied.  A classic example is the Kahn-Zhao theorem, proved initially by Kahn \cite{Kahn} in the bipartite case, and then extended to the general case by Zhao \cite{Zhao}.

\begin{thm}[Kahn-Zhao]\label{thm:KahnZhao}
    If $G$ is a $d$-regular graph then $\ind(G)$, the number of independent sets in $G$, satisfies
    \[
        \ind(G) \le \left(2^{d+1}-1\right)^{\frac{n}{2d}} = (\ind(K_{d,d}))^{\frac{n}{2d}}
    \]
    where $K_{d,d}$ is the complete balanced bipartite graph on $2d$ vertices.
\end{thm}
In particular, if $2d$ divides $n$, the $d$-regular graph with the most independent sets is a disjoint union of complete balanced bipartite graphs.  In a different vein, one could consider the independent set maximization problem for graphs having $n$ vertices and $e$ edges.    It has been shown (see, e.g.,  \cite{IndLex}) that the Kruskal-Katona Theorem \cite{Kruskal, Katona} implies that the lex graph, $\cL(n,e)$, has the greatest number of independent sets among graphs having $n$ vertices and $e$ edges.  The \emph{lex graph}, $\cL(n,e)$ is the graph that has vertex set $[n]$ and edge set the first $e$ sets in the lex (or dictionary) order, $<_L$, on $\binom{[n]}{2}$.  %Recall the lex order is given by  $A<_L B$ if $\min(A\Delta  B)\in A$ where $A\Delta B$ is the symmetric difference of $A$ and $B$.

It is natural to try to extend these extremal results for the number of independent sets to hypergraphs.  A \emph{hypergraph} $\cH$ is an ordered pair $(\cV(\cH), \cE(\cH))$ where $\cV(\cH)$ is a vertex set and $\cE(\cH)$ is a set of edges where each edge is a subset of $\cV(\cH)$. Typically we abuse notation and refer to a hypergraph as its edge set, writing, for example, $E\in\mathcal{H}$ to mean $E \in \cE(\cH)$ and $\mathcal{H} + E$ to mean $(\cV(\cH),\cE(\cH)\cup \{E\})$.  A hypergraph is \emph{$r$-uniform} if all edges have size $r$.  For convenience we'll often call an $r$-uniform hypergraph an \emph{$r$-graph}.

In a graph, an independent set is a subset of vertices containing at most one vertex from each edge.  In an $r$-graph for $r>2$, it makes sense to consider allowing more than one vertex from the independent set to be in each edge.  

\begin{defi}
	For an $r$-graph $\cH = (\cV,\cE)$ and an integer $s$ with $1\leq s\leq r$, a set $I\subset\cV$ is \emph{$s$-independent} if $\lvert I\cap E\rvert<s$ for all $E\in \cE$.  We let $\cI_{s}(\cH)$ denote the set of $s$-independent sets of a hypergraph $\cH$ and set $i_s(\cH) = \abs{\cI_s(\cH)}$.  
\end{defi}  

There has been some research on independent sets in hypergraphs, mostly focused on determining algorithms for finding independent sets in hypergraphs (see, e.g., \cite{CountingIndSets}) or on finding the independent set of largest size (see, e.g., \cite{IndSetsInHypergraphs}).  However, some extremal questions about the number of independent sets in hypergraphs have been addressed. In \cite{HypergraphIndSets} Cutler and Radcliffe give an asymptotically best possible upper bound on the number of $s$-independent sets in an $r$-uniform hypergraph of fixed size and order.  Since they use a version of the hypergraph regularity lemma, their results only apply to graphs with a large number of vertices. 

It is also the case that maximizing $1$-independent sets and $r$-independent sets in $r$-uniform hypergraphs with $n$ vertices and $e$ edges is straightforward. Defining the \emph{lex $r$-graph} $\cL_r(n,e)$ to be the $r$-graph with vertex set $[n]$ and edge set the first $e$ sets in the lex ordering\footnote{The lex ordering, $<_L$, on $\binom{[n]}{r}$ is defined by $A<_L B$ if and only if $\min\{A\Delta B\}\in A$.}, on $\binom{[n]}{r}$, the Kruskal-Katona Theorem implies the following:
\begin{thm}\label{thm:Sind}
Let $i_r(\cH)$ be the number of $r$-independent sets in $\cH$.
If $\cH$ is an $r$-uniform hypergraph with $n$ vertices and $e$ edges then
$$i_r(\cH)\leq i_r(\cL_r(n,e)).$$
\end{thm}

The \emph{colex $r$-graph} $\cC_r(n,e)$ is the $r$-graph with vertex set $[n]$ and edge set the first $e$ sets in the colex order\footnote{The colex ordering, $<_L$, on $\binom{[n]}{r}$ is defined by $A<_L B$ if and only if $\max\{A\Delta B\}\in B$.}, on $\binom{[n]}{r}$.

\begin{thm}\label{thm:1ind}
If $\cH$ is an $r$-graph on $n$ vertices with $e$ edges then
$$i_1(\cH) \leq i_1(\cC_r(n,e))$$
\end{thm}

This theorem follows immediately from the simple lemma below. 

\begin{lem}\label{lem:IsolatesAre1ind}
	For a hypergraph $\cH$ let $S(\cH)$ be the set of isolated vertices in $\cH$, and  let $s(\cH) = |S(\cH)|$.
	\begin{enumerate}
		\item $i_1(\cH)= 2^{s(\cH)}$.
		\item If $\cH\in \cH_r(n,e)$ then $s(\cH) \le s(\cC_r(n,e))$.
	\end{enumerate}
\end{lem}
\begin{proof}
For the first, note that a set $D$ is $1$-independent in a hypergraph $\cH$ if and only if $|A\cap E|<1$ for all $E\in \cE(\cH)$, i.e. $A\subseteq S(\cH)$.  Thus $i_1(\cH) = 2^{s(\cH)}$. For the second, note that trivially $s(\cH)\ge m$ requires $e\le \binom{n-m}{r}$, so
\[
	s(\cH) \le \max\setof[\big]{m}{e\le \binom{n-m}{r}}.
\]
On the other hand $\cC_r(n,e)$ achieves the bound on the right.
\end{proof}

\begin{rmk}
If $e$ is not of the form $\binom{k}{r}$ for any $k$ then there are many graphs having the same number of isolated vertices as the colex graph.  In fact, if $\binom{k-1}{r} <e<\binom{k}{r}$ then any $e$-subset of $\binom{K}{r}$ for $K$ a $k$-set has the maximum number of isolated vertices.
\end{rmk}

\subsection{Our problem} % (fold)
\label{sub:our_problem}

The problem we consider in this paper can be phrased in two ways. If we write $\cH_{r}(n,e)$ for the family of $r$-uniform hypergraphs with $n$ vertices and $e$ edges, then  from one perspective we are are determining
\[
    \max\setof[\big]{i_2(\cH)}{\cH \in \cH_3(n,e)}
\]
for all values of $n$ and $m$. The other perspective is a graph-theoretic one. If $\cH$ is a $3$-uniform hypergraph on vertex set $V$ we can consider the graph $G = \partial_2\cH$ with edge set
\[
	E(G) = \setof[\Big]{xy \in \binom{V}{2}}{\exists\, F\in E(\cH)\ \text{s.t.\ } xy \of F}.
\]
A set $I\of V$ is $2$-independent in $\cH$ if and only if does not overlap with any edge of $\cH$ in at least $2$ vertices. But this is precisely the same as requiring that $I$ is an independent set of $G$. Each edge of $\cH$ gives a triangle in $G$ (though not necessarily \emph{vice versa}). From this perspective we are trying to determine
\[
    \max\setof[\big]{i(G)}{n(G)=n,\ k_3(G)\ge e},
\]
where we write $k_3(G)$ for the number of triangles in $G$. For completeness we carefully prove the equivalence of these two problems. 
\begin{lem}
	For all $n,m\in \N$ we have
	\begin{multline*}
		\max\setof{i_2(\cH)}{\text{$\cH$ is a $3$-uniform hypergraph on vertex set $[n]$ with $e(\cH)=e$}} \\
			= \max\setof{i(G)}{\text{$G$ is a graph on vertex set $[n]$ with $k_3(G)\ge e$}}.
	\end{multline*}
\end{lem}

\begin{proof}
	To prove that the  left hand side is at most the right we just take $\cH$ to attain the maximum on the left and let $G=\partial_2\cH$. We have $k_3(G) \ge e(\cH)=e$ and $i(G) = i_2(\cH)$. In the other direction, take a graph $G$ maximizing the right hand side. Let $K_3(G)$ be the $3$-uniform hypergraph on $[n]$ whose edges are the vertex sets of triangles in $G$. By hypothesis $e(K_3(G)) \ge e$, so we can take $\cH$ to be an arbitrary spanning sub-hypergraph of $K_3(G)$ having exactly $e$ edges. We  get
	\[
		i_2(\cH) = i(\partial_2 \cH) \ge i(G), 
	\]
	since $\partial_2\cH$ is a spanning subgraph of $G$. 
\end{proof}

Phrased in this way some of the difficulties of the problem are laid bare. To find the $2$-independent sets of $\cH$ of size $t$ we need to first take the lower shadow of $\cH$ to find $G$, and then take the upper shadow of $E(G)$ on level $t$; the $2$-independent sets are those not in this upper shadow. The twin demands on $\cH$ of having not too large a lower shadow $G$, which in turn has not too large an upper shadow $\partial^t G$, are in conflict. For $\cH$ to have small lower shadow, it should look as much like a colex initial segment as possible. For $G$ to have small upper shadow it should look as much like the lex graph as possible. 

We state here our main theorem, using some undefined terms that will be clarified later and giving less detail than we do in later sections. 

\begin{mainthm}
    With a finite number of persistent exceptions (that appear for all values of $n$), and a finite number of transient exceptions (that only appear for $n\le 31$) the maximum number of independent sets in a graph $G$, subject to having at least $m$ triangles, is achieved either by the lex graph with the fewest edges subject to having at least $m$ triangles, or the \emph{lexish} graph with the fewest edges subject to having at least $m$ triangles.
    
    Equivalently, and subject to the same exceptions, the maximum number of $2$-independent sets in a $3$-uniform hypergraph with $e$ edges on $n$ vertices is achieved either by the $(2,3,1)$-lex hypergraph or the $(2,3,1)$-\emph{lexish} hypergraph having $e$ edges. \qedhere
\end{mainthm}

We have chosen in this paper to present the hypergraph as our fundamental object for the purposes of proving the main theorem. Later we will meet the \emph{downset} associated with a shifted hypergraph $\cH$. This is (essentially) the edge set of $G=\partial_2 \cH$. 

%\todo[inline]{Remind me why we say essentially here? - LK; because we don't show the adjacencies to $0$; they can be inferred: $0$ is joined to every vertex of positive degree.}

We introduce $\pi$-lex uniform hypergraphs (for any permutation $\pi$) in Section \ref{sec:pilex}.
In Section \ref{sec:results} we state our main theorem more explicitly  (Theorem \ref{thm:2indLARGE}).

We begin the proof of Theorem \ref{thm:2indLARGE} in Section \ref{sec:shifted} by providing background on shifted hypergraphs and proving that an $r$-graph attaining the maximum number of $s$-independent sets can be found among the shifted hypergraphs.  In Section \ref{sec:counting2indsets} we introduce a way to draw a shifted 3-graph as a ``nice" subset of a 3-dimensional cube and discuss a way to count the number of $2$-independent sets lost when an edge is added to a shifted 3-graph.  Using this we restate the problem yet again, in language useful for our proof. In Sections \ref{sec:moves} and \ref{sec:narrow} we introduce a set of local moves that do not decrease the number of 2-independent sets.  In Sections \ref{sec:Extensions} and \ref{sec:UBonN} we use these lemmas to determine which cases are left to prove by computation.  Finally, we prove Theorem \ref{thm:2indLARGE} in Section \ref{sec:proof}. 

% subsection our_problem (end)

\subsection{Conventions} % (fold)
\label{sub:conventions}
	We describe here some conventions that apply throughout our paper.
	\begin{itemize} 
		\item  It will be convenient for us to use a slightly non-standard ground set for our hypergraphs: we let $[n]=\set{0,1,\dots,n-1}$, and we will consider all our hypergraphs to have vertex set $[n]$ for some $n$. 
		\item We will often need to describe finite sets of integers by listing their elements. Whenever we do so we do so in increasing order. Thus when we write $A = \set{a_1, a_2, \dots, a_{k}}$ we will always assume that $a_1<a_2<\dots<a_k$.
	\end{itemize}
% subsection conventions (end)
\section{Orderings on $k$-sets and $\pi$-lex Graphs}
\label{sec:pilex}
 In order to state our  results we need to describe a number of orderings on $r$-sets of integers and some associated $r$-graphs. These graphs are an extension of the idea of lex and colex graphs to $r$-graphs for $r>2$.   Recall that the \emph{lex order}, $<_L$, on finite subsets of $\N$ is defined by $A<_L B$ if  $\min(A\Delta B)\in A$. The \emph{colex order}, $<_C$, is defined by $A<_C B$ if $\max(A\Delta B) \in B$.
We create the \emph{lex $r$-graph}, $\cL_r(n,e)$, is the $r$-graph with vertex set $[n]$ and edge set the initial segment in the lex order on $\binom{[n]}{r}$ of length $e$.  Similarly, the \emph{colex $r$-graph}, $\cC_r(n,e)$, is the $r$-graph with vertex set $[n]$ and edge set the initial segment in the colex order on $\binom{[n]}{r}$ of length $e$.

\begin{ex}
The first few edges in the lex ordering on $\binom{[n]}{2}$ are
$$\set{0,1},\set{0,2},\dots,\set{0,n-1},\set{1,2},\set{1,3},\dots,\set{1,n-1},\set{2,3},\dots$$
and the first few edges in the colex ordering on are
$$\set{0,1},\set{0,2},\set{1,2},\set{0,3},\set{1,3},\set{2,3},\set{0,4},\set{1,4},\dots$$
\end{ex}

Note that initial segments of colex do not depend on the size of the ground set, unlike those of the lex ordering.  Sets that are early  in the lex ordering have small least elements, and sets that are early in the colex ordering  have small greatest elements.  This idea will help in understanding $\pi$-lex graphs.

In $r$-graphs for $r>2$ we can define other natural orders on $\binom{[n]}{r}$ leading to other $r$-graphs.  In fact, we can define $r!$ orderings.  While these orderings seem very natural we have not seen them introduced elsewhere.

\begin{defi}\label{defi:pilex}
Consider a permutation $\pi=(\pi_1,\dots,\pi_k)$ and let $A = \{a_1,a_2,\dots, a_k\}$ and $B = \{b_1,b_2,\dots, b_k\}$ be sets in $\binom{[n]}{k}$. We define the \emph{$\pi$-lex order} on $\binom{[n]}{k}$ by $A<_{\pi} B$ if for the least $i$ for which $a_{\pi_i} \neq b_{\pi_i}$ we have $a_{\pi_i} < b_{\pi_i}$.

Given a permutation $\pi$, define the \emph{$\pi$-lex $r$-graph with $n$ vertices and $e$ edges} to be the $r$-graph on vertex set $[n]$ with edge set forming an initial segment of the $\pi$-lex order on $\binom{[n]}{r}$ of length $e$. 
\end{defi}

\begin{ex}
 The lex ordering on $\binom{[n]}{3}$ is $\pi$-lex for $\pi = (1,2,3)$ and the colex ordering on $\binom{[n]}{3}$ is $\pi$-lex for $\pi = (3,2,1)$.  
The $\pi$-lex ordering that will be particularly important to us is the $(2,3,1)$-lex ordering.
The first few sets in the $(2,3,1)$-lex ordering on $\binom{[n]}{3}$ are
\begin{align*}
\{0,1,2\}, \{0,1,3\},&\dots, \{0,1,n-1\},\{0,2,3\},\{1,2,3\},\{0,2,4\}, \{1,2,4\},\dots\\
&\{0,2,n-1\},\{1,2,n-1\},\{0,3,4\},\{1,3,4\}, \{2,3,4\}, \{0,3,5\},\\
& \{1,3,5\}, \{2,3,5\},\dots, \{0,3,n-1\},\{1,3,n-1\},\{2,3,n-1\},\\
&\{0,4,5\},\{1,4,5\},\{2,4,5\},\{3,4,5\},\dots, \{0,4,n-1\},\\
&\{1,4,n-1\},\{2,4,n-1\},\{3,4,n-1\},\dots
\end{align*}

Notice that sets that are small in the $(2,3,1)$-lex ordering have their second greatest element being small.
\end{ex}

There is a natural partial ordering on $\binom{[n]}{k}$ that will also be relevant, that we call the \emph{compression} ordering. Given $A=\set{a_1, a_2, \dots, a_{k}}$ and $B=\set{b_1, b_2, \dots, b_{k}}$ we let $A \comple B$ if $a_i \le b_i$ for all $i$. Equivalently, $A\comple B$ if and only if for all $x\in \R$ we have $\abs{A\cap (-\infty,x]} \ge \abs{B\cap (-\infty,x]}$. The following simple lemma will be useful later.

\begin{lem}\label{lem:union}
	If $A_1, B_1$ are  $s$-sets with $A_1\comple B_1$, $A_2,B_2$ are  $(r-s)$-sets with $A_2\comple B_2$ and $A_1\cap A_2 = B_1\cap B_2 = \emptyset$ then 
	$A\comple B$ where $A=A_1\cup A_2$ and $B= B_1\cup B_2$.
\end{lem}

\begin{proof}
	For all $x\in \R$ we have
	\begin{align*}
		\abs{A\cap (-\infty,x]} &= \abs{A_1\cap (-\infty,x]} + \abs{A_2\cap (-\infty,x]} \\
			&\ge \abs{B_1\cap (-\infty,x]}  + \abs{B_2\cap (-\infty,x]}  = \abs{B\cap (-\infty,x]}.
	\end{align*}
\end{proof}

\section{Shifted Hypergraphs}
\label{sec:shifted}
Since threshold graphs appear as an answer to many extremal questions in graphs, the concept of a ``threshold hypergraph" should be useful when answering similar questions in hypergraphs.  While there are many equivalent definitions of    threshold graphs (see \cite{OrangeBook}),  in \cite{ThresholdHypergraphs} Reiterman, R\"odl, \v{S}i\v{n}ajov\'{a}, and T\r{u}ma show that the extensions of three of the equivalent definitions of threshold graphs are not equivalent for $r$-graphs with $r>2$.  The version that will be useful to us is the notion of \emph{shifted hypergraphs}, introduced in \cite{Frankl}.  We will show that $s$-independent sets in $r$-graphs are maximized by shifted hypergraphs and use this fact restate the problem.

\begin{defi}
Given a set $A\subset [n]$ and $i,j\in [n]$ such that $A\cap \{i,j\} = \{i\}$ define $A_{i\to j} = (A\setminus\{i\})\cup\{j\}$.
\end{defi}

\begin{defi}
Consider a hypergraph $\cH$ with vertex set $[n]$ and edge set $\cE$. For $0\leq j <i \leq n-1$ define the \emph{$(i,j)$-shift} $S_{i\to j}$ as follows:
\begin{itemize}
\item for each $E\in \cE$, 
$$S_{i\to j}(E) = \begin{cases}
E_{i\to j} &\text{if $E\cap\{i,j\} = \{i\}$}\\
E &\text{otherwise}
\end{cases}.$$
\item let $S_{i\to j}(\cE) = \{S_{i\to j}(E): E\in \cE\}\cup\{E: E,S_{i\to j}(E) \in \cE\}$.
\end{itemize}
For a hypergraph $\cH$ on vertex set $[n]$, we will write $\cH_{i\to j}$ to mean the hypergraph on vertex set $[n]$ with edge set $S_{i\to j}(\cE(\cH))$.  
\end{defi}

Thus, $\cH_{i\to j}$ is a hypergraph with the same number of edges as $\cH$ with the same sizes, but where we have replaced $i$ with $j$ whenever possible.  

\begin{defi}
A hypergraph $\cH= ([n],\cE)$ is \emph{shifted} if and only if $\cH_{i\to j} = \cH$ for all $0\leq j<i\leq n-1$.
\end{defi}

We will extend the definition of $\cH_{i\to j}$ slightly and set $\cH_{i\to i} = \cH$ for all $i\in [n]$.  In the next definition we extend again to apply a number of shifts at once. 
\begin{defi}
	Given an $r$-graph $\cH$ and $k$-sets $A \comple B$ with $A = \set{a_1,a_2,\dots,a_k}$, $B=\set{b_1, b_2, \dots, b_{k}}$ we define
	\[
		\cH_{B\to A} = (\cdots((\cH_{b_1\to a_1})_{b_2\to a_2})\cdots)_{b_k\to a_k}.
	\]
\end{defi}

We will use this definition in Section \ref{sec:counting2indsets}.  In particular, we will use the fact that if we apply a shift from all the vertices in one edge to another $r$-set of vertices, $A$, then $A$ will be in the edge set of the shifted graph.  We prove this in the next lemma.

\begin{lem}\label{lem:vectorshift}
	If $\cH$ is an $r$-graph on $[n]$ and $A\comple B$ are $r$-sets with $B\in \cH$ then $A\in \cH_{B\to A}$.
\end{lem}

\begin{proof}
	We'll prove it by (reverse) induction on the parameter 
	\[
		\ell=\max\setof{j}{\text{$a_i=b_i$ for all $i\le j$}}. 
	\]
	If $\ell=r$ then $A=B$ and there is nothing to prove. If $\ell=r-1$ then $A = B_{b_r\to a_r}$ and $\cH_{B\to A} = \cH_{b_r\to a_r}$. It is clear from the definition of shifting that $A\in \cH_{B\to A}$. Suppose then that $\ell < r-1$. Note that $a_{\ell+1}\neq b_{\ell+1}$. Consider $B' = B \symd \set{a_{\ell+1},b_{\ell+1}} = B_{b_{\ell+1}\to a_{\ell+1}}$. We have $A\comple B'\comple B$. Since all earlier compressions have no effect we have $\cH_{B\to A} = (\cH_{b_{\ell+1}\to a_{\ell+1}})_{B'\to A}$. By the definition of shifting we know that $B' \in \cH_{b_{\ell+1}\to a_{\ell+1}}$ since $B\in\cH$. This implies by induction that $A\in (\cH_{b_{\ell+1}\to a_{\ell+1}})_{B'\to A}=\cH_{B\to A}$, as required.
\end{proof}

\subsection{Shifted Hypergraphs Maximize $s$-independent Sets}
In this section we will show that for any $r,s,n,$ and $e$ we can find a $r$-graph maximizing the number of $s$-independent sets in $\cH_r(n,e)$ among the shifted hypergraphs.  In the next proof we will construct an injection from the set of $s$-independent sets in some hypergraph $\cH$ to the set of $s$-independent sets in the shift $\cH_{i\to j}$.  Note that in the next lemma we need not assume that the hypergraph is uniform.

\begin{lem}
\label{lem:shift}
Let $\cH$ be a hypergraph with vertex set $[n]$ and let $0\leq j < i <n$.  Then for all $s$,
$$i_s(\cH_{i\to j}) \geq i_s (\cH).$$
\end{lem}
\begin{proof}
We will define an injection  from $\cI_s(\cH)\setminus \cI_s(\cH_{i\to j})$ to $\cI_s(\cH_{i\to j}) \setminus \cI_s(\cH).$
Let $I$ be an independent set in $\cI_s(\cH)\setminus\cI_s(\cH_{i\to j})$.  
If $j\notin I$ we have $|I \cap S_{i\to j}(E)|\leq |I\cap E|$ for all $E\in \cE$ and so $j\in I$. Similarly, $i\notin I$, because if $I$ is $s$-independent in $\cH$ and $i,j\in I$ then $I$ is $s$-independent in $\cH_{i\to j}$.  Define $f: \cI_s(\cH)\setminus \cI_s(\cH_{i\to j}) \to \cI_s(\cH_{i\to j}) \setminus \cI_s(\cH)$ by $f(I) = I_{j\to i}$.  
This is clearly an injection so we need only show that  $I_{j\to i} \in \cI_s(\cH_{i\to j}) \setminus \cI_s(\cH)$.  Let $F\in \cE(\cH_{i\to j})$ and consider $|I_{j\to i}\cap F|$.  

Recall $\cE(\cH_{i\to j}) = \{S_{i\to j}(E): E \in \cE(\cH)\} \cup \{E: E, S_{i\to j}(E) \in  \cE(\cH)\}$.  
Suppose $F\in \{S_{i\to j}(E): E \in \cE(\cH)\}$.  Then either
\begin{itemize}
\item $F = E$ for some $E\in \cE(\cH)$ because  $E\cap \{i,j\} \neq \{i\}$ and so $S_{i\to j}(E) = E$ or
\item $F = E_{i\to j}$ for some $E\in\cE(\cH)$
\end{itemize}

Suppose $F\in\{E: E,S_{i\to j}(E) \in \cE(\cH)\}$.  It's possible that $E$ and $S_{i\to j}(E)$ are in $\cE(\cH)$ for two reasons:
\begin{itemize}
\item  $S_{i\to j}(E) = E$ because $E\cap\{i,j\} \neq \{i\}$ (which is the same as the first case above) or
\item $S_{i\to j}(E) = E_{i\to j}$ but $E_{i\to j} \in \cE(\cH)$
\end{itemize}

So the proof will be in three cases.
\begin{enumerate}
\item Suppose that $F= E$ for some $E\in\cE(\cH)$ such that $E\cap \{i,j\} \neq \{i\}$.  If $E\cap\{i,j\} = \emptyset$ then
$$|I_{j\to i} \cap F| = |I_{j\to i} \cap E| = |I\cap E|<s.$$
If $E\cap \{i,j\} = \{j\}$ then
$$|I_{j\to i} \cap F| = |I_{j\to i} \cap E| < |I\cap E| <s.$$
If $E\cap \{i,j\} = \{i,j\}$ then 
$$|I_{j\to i} \cap F| = |I_{j\to i} \cap E| = |I\cap E| <s.$$

\item  Suppose that $F = E_{i\to j}$ for some $E\in\cE(\cH)$.  Then
$$|F\cap I_{j\to i}| = |E_{i\to j}\cap I_{j\to i}| =|E\cap I| <s.$$

\item  Suppose that $F = E$ for some $E\in \cE(\cH)$ such that $E\cap\{i,j\}=\{i\}$ and $E_{i\to j}\in\cE(\cH)$.  Then
$$|F\cap I_{j\to i}| = |E\cap I_{j\to i}| = |E_{i\to j} \cap I| <s.$$
\end{enumerate}

Therefore $I_{j\to i}\in \cI_s(\cH_{i\to j})$.  It remains to show that $I_{j\to i}\notin \cI_s(\cH)$.  Since $I\notin \cI_s(\cH_{i\to j})$ there exists $E\in \cE(\cH_{i\to j})$ such that $|I\cap E|\geq s$.  It must be the case that $E = F_{i\to j}$ for some $F\in \cH$ and $E\neq F$.  Then
$$s\leq |I\cap E| = |I_{j\to i} \cap E_{j\to i}| = |I_{j\to i} \cap F|.$$
Thus, $I_{j\to i} \notin \cI_s(\cH)$.  So, $|\cI_s(\cH)\setminus \cI_s(\cH_{i\to j})|\leq |\cI_s(\cH_{i\to j}) \setminus \cI_s(\cH)|.$  Therefore, 
$$|\cI_s(\cH)|\leq |\cI_s(\cH_{i\to j})|.$$
\end{proof}

\begin{cor}
\label{cor: shifted}
A hypergraph maximizing the number of $s$-independent sets among all hypergraphs with $n$ vertices and $e$ edges can be found among the shifted hypergraphs. 
\end{cor}
\begin{proof}
Let $\displaystyle{t(\cH) = \sum_{E\in\cE(\cH)} \sum_{i\in E} i}$.  Pick $\cH$ with the maximal number of $s$-independent sets and $t(\cH)$ minimal.  Let $0\leq j <i\leq n$.  Note $\cH_{i\to j}$ has the same number of vertices and edges as $\cH$ and $i_s(\cH_{i\to j}) \geq i_s(\cH)$ by Lemma \ref{lem:shift}.  Thus, we must have $\cH_{i\to j} = \cH$, else $t(\cH_{i\to j})<t(\cH)$ contradicting the definition of $\cH$.  So $\cH$ is a shifted hypergraph maximizing the number of $s$-independent sets.
\end{proof}

For the remainder of the paper we will focus on shifted hypergraphs.\section{Formal Statement of Main Result}
\label{sec:results}

Theorems \ref{thm:Sind} and \ref{thm:1ind} answer the question of which $3$-graphs have the most $3$-independent sets and $1$-independent sets, respectively. Our main result answers the question of which $3$-graphs have the most $2$-independent sets. We need some preliminary definitions before we state the theorem. 

 As shown in Section \ref{sec:shifted}, we need only consider shifted hypergraphs. It will turn out that the feature of a shifted $3$-graph $\cH$ that determines $i_2(\cH)$ is the collection of its edges that contain $0$. We make the following definition so that we can state our main result, but we discuss the topic more extensively in Section \ref{sec:counting2indsets}.

\begin{defi}
	Given a shifted $3$-graph $\cH$ the \emph{downset} of $\cH$ is the set 
	\[
		D(\cH) = \setof{(i,j)}{\set{0,i,j}\in \cH}.
	\]
	This is indeed a downset in the poset 
	\[
		B_n = \setof{(i,j)}{1\le i<j \le n-1} \of \set{1,2,\dots,n-1}^2
	\]
	with the product order.
\end{defi}

Associating hypergraphs to downsets is a many to one relationship.  A hypergraph $\cH$ has exactly one downset, but given a downset $D$, there are often many (shifted) hypergraphs that have downset $D$. An example of how we visualize the downset is shown in Figure \ref{fig:downsetvis}. A cell $(i,j)$ is shaded provided that $\{0,i,j\}\in \cH$. The downset of a hypergraph differs from the lower shadow $\partial_2(\mathcal{H})$ introduced in Section \ref{sub:our_problem} in that the edges in $\partial_2(\mathcal{H})$ that contain $0$ are not shown in the downset---they are implied.

\begin{figure}
\begin{center}
\begin{tikzpicture}[scale=.9]
\draw[step=1cm,black,very thin] (0,0) grid (1,5);
\draw[step=1cm,black,very thin] (1,1) grid (2,5);
\draw[step=1cm,black,very thin] (2,2) grid (3,5);
\draw[step=1cm,black,very thin] (3,3) grid (4,5);
\draw[step=1cm,black,very thin] (4,4) grid (5,5);
\draw[fill=lightgray] (0,0)--(0,1)--(1,1) -- (1,0) -- (0,0);
\draw[fill=lightgray] (0,1)--(0,2)--(1,2) -- (1,1) -- (0,1);
\draw[fill=lightgray] (0,2)--(0,3)--(1,3) -- (1,2) -- (0,2);
\draw[fill=lightgray] (0,2)--(0,3)--(1,3) -- (1,2) -- (0,2);
\draw[fill = lightgray] (1,1) -- (1,2) -- (2,2) -- (2,1) -- (1,1);
\draw (.5,.5) node {\tiny{$(1,2)$}};
\draw (.5,1.5) node {\tiny{$(1,3)$}};
\draw (.5,2.5) node {\tiny{$(1,4)$}};
\draw (.5,3.5) node {\tiny{$(1,5)$}};
\draw (.5,4.5) node {\tiny{$(1,6)$}};
\draw (1.5,1.5) node {\tiny{$(2,3)$}};
\draw (1.5,2.5) node {\tiny{$(2,4)$}};
\draw (1.5,3.5) node {\tiny{$(2,5)$}};
\draw (1.5,4.5) node {\tiny{$(2,6)$}};
\draw (2.5,2.5) node {\tiny{$(3,4)$}};
\draw (2.5,3.5) node {\tiny{$(3,5)$}};
\draw (2.5,4.5) node {\tiny{$(3,6)$}};
\draw (3.5,3.5) node {\tiny{$(4,5)$}};
\draw (3.5,4.5) node {\tiny{$(4,6)$}};
\draw (4.5,4.5) node {\tiny{$(5,6)$}};
\end{tikzpicture}
\caption{A visualization for a downset for a hypergraph with $7$ vertices. The hypergraph could have edge set $\left\{\{0,1,2\}, \{0,1,3\}, \{0,1,4\}, \{0,2,3\}\right\}$ or  $\left\{\{0,1,2\}, \{0,1,3\}, \{0,1,4\}, \{0,2,3\}, \{1,2,3\}\right\}$.
}
\label{fig:downsetvis}
\end{center}
\end{figure}

In Section \ref{sec:pilex} we introduced $(2,3,1)$-lex $3$-graphs. The maximizers of $2$-independent sets in $\mathcal{H}_3(n,e)$ are generally $(2,3,1)$-lex graphs. We describe the $3$-graphs that are maximizers by their downsets in the following definition.

\begin{defi}
We say that a shifted 3-graph $\cH$ is \emph{$(2,3,1)$-lex style} if its downset $D = D(\cH)$ satisfies
\begin{itemize}
\item $D$ is an initial segment in lex order, or
\item $D$ is a downset in $B_n$ that is an initial segment in lex order missing one edge.
\end{itemize}
\end{defi}

The possible downsets of $(2,3,1)$-lex style $3$-graphs are shown in Figure \ref{fig:231lex}.

\begin{rmk}
	All $(2,3,1)$-lex graphs are $(2,3,1)$-lex style as a consequence of having a downset that is an initial segment in lex order have the property that we can arrange the edges not in the base layer so that they form an initial segment in $(2,3,1)$-lex order.   Notice that if $D$ is a downset in $B_n$ that is an initial segment in lex order missing one edge then that edge must correspond to the top cell in the second to last column.  This is shown in the right downset in Figure \ref{fig:231lex}.  

\end{rmk}

 Theorem \ref{thm:2indLARGE} says, roughly, that hypergraphs that have downsets that are $(2,3,1)$-lex style maximize $2$-independent sets.  In the following theorem we describe the non-$(2,3,1)$-lex style hypergraphs that maximize $2$-independent sets by their lower shadow graph.

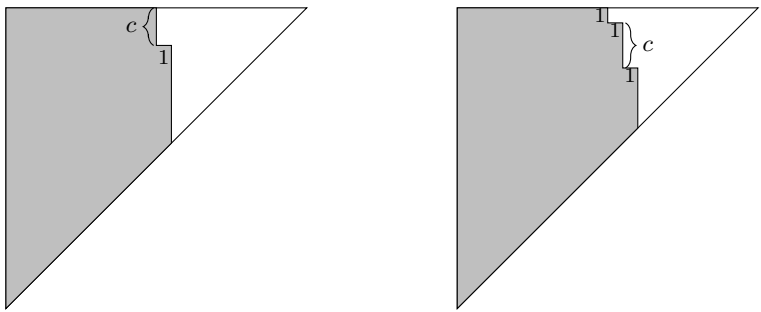
\begin{figure}
\begin{center}
\begin{tikzpicture}[scale=1]
\draw (0,0) -- (0,4) -- (4,4) -- cycle;
\draw[fill=lightgray] (0,0) -- (0,4) -- (2,4) -- (2,3.5) -- (2.2,3.5) -- (2.2,2.2) -- cycle;
\draw (2.1,3.35) node {\tiny{$1$}};
\draw [decorate,decoration={brace,amplitude=4pt},xshift=-1pt,yshift=0pt] (2,3.5) -- (2,4) node [black,midway,xshift=-0.3cm]{\scriptsize{$c$}}; 
\draw (6,0) -- (6,4) -- (10,4) -- cycle;
\draw[fill=lightgray] (6,0) -- (6,4) -- (8,4) -- (8,3.8) -- (8.2,3.8) -- (8.2,3.2) -- (8.4,3.2)  -- (8.4,2.4)-- cycle;
\draw (7.9,3.9) node {\tiny{$1$}};
\draw (8.1,3.7) node {\tiny{$1$}};
\draw (8.3, 3.1) node {\tiny{$1$}};
\draw [decorate,decoration={brace,amplitude=4pt,mirror},xshift=1pt,yshift=0pt] (8.2,3.2) -- (8.2,3.8) node [black,midway,xshift=0.3cm]{\scriptsize{$c$}}; 
\end{tikzpicture}

\caption{The (zoomed out) downset of a $(2,3,1)$-lex 3-graph at left and a $(2,3,1)$-lex style $3$-graph that is not $(2,3,1)$-lex at right.  The vertical drops are length $c$ for some $c\geq 0$.}
\label{fig:231lex}
\end{center}
\end{figure}

\begin{thm}
\label{thm:2indLARGE}
Let $\cH$ be a $3$-graph on $n$ vertices with $e$ edges where $n\geq 32$. Then there exists a $3$-graph $\cG$ with $n$ vertices and $e$ edges such that
$$i_2(\cH) \leq i_2(\mathcal{G}),$$
where  $\cG$ is either $(2,3,1)$-lex style or $\cG$ has $\partial_2(\cG)$ coming from one of the following set of $5$ persistent exceptions: 
\[\mathcal{P}_n = \{(K_3\vee E_1)\cup E_{n-5}, (K_2\vee E_{n-5})\cup E_2, (K_2\vee E_{n-4})\cup E_1, K_3\vee E_{n-4} ,K_4\vee E_{n-5}\}.\]
 When $n<32$ there are $16$ possible downsets of hypergraphs that maximize $2$-independent sets that are not $(2,3,1)$-lex style or in $\mathcal{P}_n$. These downsets are shown in Table \ref{table:exceptions}.

\end{thm}

To complete the picture, we state the equivalent theorem for the graph problem. We need a definition first.
\begin{defi}
    A graph with $n$ vertices and $e$ edges is \emph{lexish} if it is either the lex graph $\cL(n,e)$ or else $\cL(n,e)-f$ where $f$ is the edge $(i-1) n$ where $i$ is such that $\set{1,2,\dots,i+1}$ is the unique largest clique in $\cL(n,e)$. 
\end{defi}

\begin{thm}
    Let $H$ be a graph on $n$ vertices with $t$ triangles where $n\ge 32$. Then there exists a graph $G$ on $n$ vertices such that $k_3(G)\ge t$ and $i(G)\ge i(H)$ and moreover $G$ is either a lex graph, a lexish graph, or $(K_2 \vee E_t) \cup E_{n-t-2}$.
\end{thm}
%For nearly all of these counterexamples, $\partial(\mathcal{H})$ is colex, with most of those being complete graphs.  These examples are presented in the following table:
\begin{table}
\begin{center}
\begin{tabular}{|c|c|}
\hline
$n$ &$\partial_2(\mathcal{H})$\\
\hline
$7$ &$K_5$\\
$8$ &$K_5$, $K_6$, $K_7$\\
$9$ &$K_5$, $K_6$, $K_7$, $K_8$, $K_9-K_{1,6}$\\
$10$ &$K_9$ \\
$11$ &$K_{10}$, $K_{11}-K_{1,9}$\\
$12$ &$K_{11}$\\
$14$ &$K_{13}$, $K_{13}-e$\\
$16$ &$K_{15}$\\
\hline
\end{tabular}
\end{center}
\caption{All exceptions to the maximizer being $(2,3,1)$-lex style when $n<32$.}
\label{table:exceptions}
\end{table}

%\begin{rmk}
%The exceptions are all split graphs.  They are.
%\end{rmk}

\section{Counting 2-independent Sets in Shifted 3-graphs}\label{sec:counting2indsets}

In this section we will develop a way to count $2$-independent sets in shifted $3$-graphs.  This will result in a translation of the problem to an optimization problem that is easier to visualize.

\begin{defi}
Given $r\ge s\ge 2$, suppose $I\subseteq [n]$ is a set of size at least $s$. Let $I_s$ be the $s$-set consisting of the $s$ smallest elements of $I$, and let $J$ be the $r-s$ smallest elements of $[n]\wo I_s$. Define the \emph{minimal edge of $I$} to be $E_0(I) = I_s \cup J$.   Note that $E_0(I)$ is the  unique $\comple$-minimal set in $\binom{[n]}{r}$ that has $|E\cap I| \geq s$.
\end{defi}

\begin{rmk}
For $r=3, s=2$ and $I\subset[n]$ of size at least 2, the minimal edge of $I$ is $E_0(I) = \{a_1,a_2,b\}$ where $a_1$ and $a_2$ are the two smallest elements of $I$ and $b = \min\{i\in[n]: i\neq a_1,a_2\}$. 
\end{rmk}

The purpose of defining the minimal edge of a set $I$ is that $I$ is $s$-independent in a shifted $r$-graph $\cH$ exactly when $E_0(I)$ is not in $\mathcal{H}$.

\begin{lem}\label{lem:Sindcondition}
Let $\cH$ be a shifted $r$-graph and consider a set  $I\subseteq [n]$ with $|I|\geq s$.   The set $I$ is $s$-independent in $\cH$ if and only if $E_0(I)\notin \cE(\cH)$. 
\end{lem}

\begin{proof}
Suppose that $I$ is an $s$-independent set.  Then $E_0(I)\notin \cE(\cH)$ since $|I\cap E_0(I)| \geq s$. 
Suppose now that $I$ is not an $s$-independent set. There exists an edge $E\in \cH$ such that $\abs{E\cap I}\ge s$. Let $E_s$ be the set of the $s$ smallest elements of $E\cap I$, and $F$ be $E\wo E_s$. Note that, with the notation of the previous definition, $I_s\comple E_s$, since $I_s$ is the unique $\comple$-minimal $s$ set in $I$. It is also true that $J \comple F$. To see this note first that $F\of [n]\wo I_s$; any $x\in F\cap I_s$ would have to be one of the $s$ smallest elements of $E \cap I$, hence in $E_s$, a contradiction. Now $J\comple F$ since $J$ is the unique $\comple$-minimal $(r-s)$-set in $[n]\wo I_s$. By Lemma~\ref{lem:union} we have $E_0(I) = I_s\cup J \comple E_s\cup F$. Now by Lemma~\ref{lem:vectorshift}, since $E\in \cH$, we have
\[
	E_0(I) \in \cH_{E \to E_0(I)} = \cH,
\]
the last equality holding since $\cH$ is shifted.
\end{proof}

\begin{cor}
\label{cor:Sindcondition}
Let $I\subset[n]$ with $|I|\geq s$. Suppose $\cH' = \cH + E$ and that $\cH'$ and $\cH$ are shifted $r$-graphs.  Then $I\in \cI_s(\cH)\setminus \cI_s(\cH')$ if and only if $E_0(I) =E$.
\end{cor}
\begin{proof}
By Lemma \ref{lem:Sindcondition}, $I\in \cI_s(\cH)$ if and only if $E_0(I) \notin \cH$ and $I\notin \cI_s(\cH')$ if and only if $E_0(I) \in \cH'$.  Thus,  $I \in \cI_s(\cH)\setminus \cI_s(\cH')$ if and only if $E_0(I) = E = \cH'\setminus\cH$.
\end{proof}

Now we are able to calculate the number of sets that are lost when an edge is added to a shifted hypergraph.

\begin{lem}\label{lem:cost}
Let $\cH$ be a shifted 3-graph on vertex set $[n]$, let $E = \{i,j,k\}$ and suppose that $\cH' = \cH + E$ is also shifted.  Then
$$i_2(\cH') = i_2(\cH) - c_{ijk}$$
where
$$c_{ijk} = \begin{cases}
2^{n-1} &\text{if $\{i,j,k\}=\{0,1,2\}$}\\
2^{n-k} &\text{if $i=0,j=1$ and $k\neq 2$}\\
2^{n-k-1} &\text{if $i=0$ and $j> 1$}\\
0  &\text{if $i\neq 0$}
\end{cases}.$$

\end{lem}

\begin{rmk}
We will refer to $c_{ijk}$ as the \emph{cost} of the edge $\{i,j,k\}$.
\end{rmk}

\begin{proof}
By Corollary \ref{cor:Sindcondition}, $I \in i_2(\cH)\setminus i_2(\cH')$ if and only if $E_0(I) = E$.  Thus, to determine the cost of an edge $E$ we must count the number of sets $I$ such that $E_0(I) = E$.

If $E = \{0,1,2\}$ we are counting sets such that $E_0(I) = \{0,1,2\}$.  These are exactly those sets having two smallest elements $0$ and $1$, $0$ and $2$, or $1$ and $2$.  The number of sets with this property is $2^{n-2} + 2^{n-3} + 2^{n-3} = 2^{n-1}$. Thus, $c_{012} = 2^{n-1}$.  

Suppose that $\{0,1,k\}$ is added to a hypergraph where $k\neq 2$.  Here we count sets $I$ such that $E_0(I) = \{0,1,k\}$.  These are the sets with smallest elements $0$ and $k$ or $1$ and $k$.  The number of sets with this property is $2^{n-k-1}+2^{n-k-1} = 2^{n-k}$.  Thus $c_{01k} = 2^{n-k}$ for $k\neq 2$. 

Suppose now $E = \{0,j,k\}$ with $j> 1$.  Here, $E_0(I) = E$ if and only if the two smallest elements of $I$ are $j$ and $k$.  There are $2^{n-k-1}$ of these meaning $c_{0jk} = 2^{n-k-1}$ when $j> 1$.

Finally, if $0\notin E$ then it is not one of the edges of the form $E = \{a_1,a_2,b\}$ where $b = \min\{i\in [n]: i \neq a_1, a_2\}$.  Thus, the cost of $\{i,j,k\}$ where $i\neq 0$ is $0$.
\end{proof}

Note that $\displaystyle{\sum_{i<j<k} c_{ijk} = 2^{n} - (n+1)}$ meaning that $i_2(\mathcal{K}^3_n)= n+1$ where $\mathcal{K}^3_n$ is the complete $3$-graph on $n$ vertices.  The $2$-independent sets in $\mathcal{K}_n^3$ are the empty set and all the singletons. 

Let $\cH$ be a $3$-graph with vertex set $[n]$.  We will visualize $\cH$ by letting its edges be $1\times 1\times 1$ cubes labeled by the vertices in the edge in increasing order.  Then we can think of these $1\times 1\times 1$ cubes inside an $(n-2)\times (n-2) \times (n-2)$ cube labeled as in Figure \ref{fig:cube}.  Figure \ref{fig:complete} shows the edges of the complete hypergraph on $7$ vertices inside a $5\times 5\times 5$ cube with the visible cubes labeled.

\begin{multicols}{2}
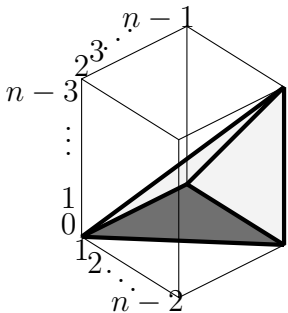
\begin{figure}[H]
\begin{center}
\begin{tikzpicture}[scale=.7]
\draw[fill = light-gray, opacity = .85] (5,4,3) -- (5,1,3) -- (0,0,0) -- cycle;
%\draw[fill = medium-gray, opacity=.85] (5,4,3) -- (5,1,3) -- (2,1,0) -- cycle;
\draw[fill = dark-gray,opacity=.85] (5,1,3) -- (2,1,0) -- (0,0,0) -- cycle;
%cube
\draw (0,0,0) -- (0,3,0) -- (3,3,3) -- (3,0,3) -- cycle;
\draw (2,1,0) -- (2,4,0) -- (5,4,3) -- (5,1,3) -- cycle;
\draw (0,0,0) -- (2,1,0);
\draw (0,3,0) -- (2,4,0);
\draw (3,3,3) -- (5,4,3);
\draw (3,0,3) -- (5,1,3);
%tetrahedron
\draw[ultra thick] (2,1,0) -- (5,4,3);
\draw[ultra thick] (0,0,0) -- (2,1,0);
\draw[ultra thick] (5,1,3) --(2,1,0);
\draw[ultra thick] (5,1,3) -- (0,0,0);
\draw[ultra thick] (5,4,3) -- (0,0,0);
\draw[ultra thick] (5,4,3) -- (5,1,3);
%labels
%first coordinate
\draw (-.25,.25,0) node {$0$};
\draw (-.25,.75,0) node {$1$};
\draw (-.25, 2,0) node {$\vdots$};
\draw (-.75, 2.75,0) node {$n-3$};
%second coordinate
\draw (0,-.25,0) node {$1$};
\draw (.25,-.5,0) node {$2$};
\draw (.75,-.65,0) node {$\ddots$};
\draw (1.25,-1.25,0) node {$n-2$};
%third coordinate
\draw (0,3.25,0) node {$2$};
\draw (0,3.25,-.75) node {$3$};
\draw (0,3.25, -1.75) node {$\iddots$};
\draw (.5,3.2, -2.5) node {$n-1$};
\end{tikzpicture}
\caption{The labeling of the cube.  The shaded tetrahedron represents the collection of $1\times 1\times 1$ cubes that have labels in increasing order.}
\label{fig:cube}
\end{center}
\end{figure}

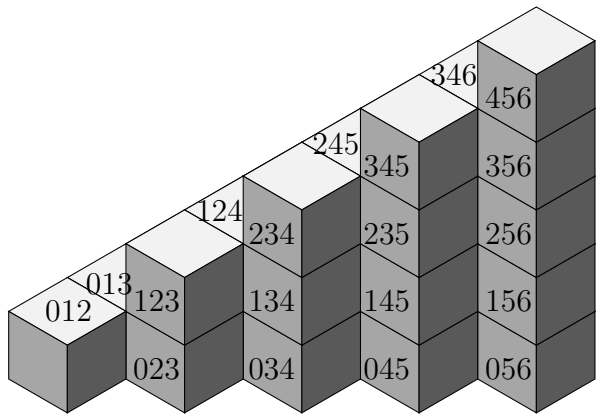
\begin{figure}[H]
\begin{center}
\begin{tikzpicture}[scale=.9]
\planepartition{{1,2,3,4,5},{1,2,3,4},{1,2,3},{1,2},{1}}
\draw (-3.6,-1.65,-0.4) node {$012$};
\draw (-3,-1.25,-0.4) node {$013$};
\draw (-2.3,-2.5,-0.4) node {$023$};
\draw (-2.3,-1.5,-0.4) node {$123$};
\draw (-1.6,-.4,-1) node {$124$};
\draw (-.6,-2.5,-0.4) node {$034$};
\draw (-.6,-1.5,-0.4) node {$134$};
\draw (-.6,-.5,-0.4) node {$234$};
\draw (1.1,-.5,-0.4) node {$235$};
\draw (1.1,-2.5,-0.4) node {$045$};
\draw (1.1,-1.5,-0.4) node {$145$};
\draw (0.35,0.83,-0.4) node {$245$};
\draw (1.1,0.5,-0.4) node {$345$};
\draw (1.9,1.65,-0.9)  node {$346$};
\draw (2.9,-2.5,-0.4) node {$056$};
\draw (2.9,-1.5,-0.4) node {$156$};
\draw (2.9,-0.5,-0.4) node {$256$};
\draw (2.9,0.5,-0.4) node {$356$};
\draw (2.9,1.5,-0.4) node {$456$};
\end{tikzpicture}
\caption{Edges of the complete hypergraph on 7 vertices.}
\label{fig:complete}
\end{center}
\end{figure}
\end{multicols}

 Lemma \ref{lem:cost} says that, assuming the hypergraph is shifted, any edge that does not contain $0$ is ``free", i.e., adding such an edge does not cost us any independent sets.  More rigorously, if $E = \{i,j,k\}$ with $i\neq 0$ we have  $i_2(\cH) = i_2(\cH + E)$.  In the cube picture this means that any edge that is not in the bottom layer is free. For this reason, we focus on the downset of $\cH$. The downset of $\cH$ corresponds to edges in the base layer.  Figure \ref{fig:baselayer} shows the cube where we have suppressed the first dimension and show only the edges with non-zero costs.

\begin{figure}[H]
\begin{center}
\begin{tikzpicture}[scale=.75]
\draw[step=1cm,black,very thin] (0,0) grid (1,5);
\draw[step=1cm,black,very thin] (1,1) grid (2,5);
\draw[step=1cm,black,very thin] (2,2) grid (3,5);
\draw[step=1cm,black,very thin] (3,3) grid (4,5);
\draw[step=1cm,black,very thin] (4,4) grid (5,5);
\draw (.5,.5) node {$012$};
\draw (.5,1.5) node {$013$};
\draw (.5,2.5) node {$014$};
\draw (.5,3.5) node {$015$};
\draw (.5,4.5) node {$016$};
\draw (1.5,1.5) node {$023$};
\draw (1.5,2.5) node {$024$};
\draw (1.5,3.5) node {$025$};
\draw (1.5,4.5) node {$026$};
\draw (2.5,2.5) node {$034$};
\draw (2.5,3.5) node {$035$};
\draw (2.5,4.5) node {$036$};
\draw (3.5,3.5) node {$045$};
\draw (3.5,4.5) node {$046$};
\draw (4.5,4.5) node {$056$};
\end{tikzpicture}
\caption{Edges in base layer, $B_7$.}
\label{fig:baselayer}
\end{center}
\end{figure}

We will call each of the squares in $B_n$ a \emph{cell} and label it $(a,b)$ if the edge associated to that square is $\{0,a,b\}$.

Recall that we are restricting ourselves to shifted hypergraphs as we can find a maximizer among the shifted hypergraphs. By definition a shifted hypergraph $\cH$ on vertex set $[n]$ satisfies the following condition: if $\{a,b,c\}\in \cE(\cH)$ then $\{i,j,k\}\in\cE(\cH)$ whenever $i\leq a$, $j\leq b$, and $k\leq c$.  In  $B_n$ this says that if $\{0,b,c\}\in \cE(\cH)$ then $\{0,j,k\}\in\cE(\cH)$ for all $j\leq b$ and $k\leq c$.  That is, if we include a cell $(b,c)$ in our hypergraph, we must also include all cells that are to the left or below. 

Each cell has an associated cost as given in Lemma \ref{lem:cost} and an associated amount of \emph{space}: the number of edges we could get for that cost, given that taking those edges results in a shifted hypergraph.  The cost and space for cells in $B_7$ are given in Figure \ref{fig:CostAndSpace}.

\begin{center}
\begin{figure}[H]
\begin{multicols}{2}
\begin{tikzpicture}[scale=.75]
\draw[step=1cm,black,very thin] (0,0) grid (1,5);
\draw[step=1cm,black,very thin] (1,1) grid (2,5);
\draw[step=1cm,black,very thin] (2,2) grid (3,5);
\draw[step=1cm,black,very thin] (3,3) grid (4,5);
\draw[step=1cm,black,very thin] (4,4) grid (5,5);
\draw (.5,.5) node {$64$};
\draw (.5,1.5) node {$16$};
\draw (.5,2.5) node {$8$};
\draw (.5,3.5) node {$4$};
\draw (.5,4.5) node {$2$};
\draw (1.5,1.5) node {$8$};
\draw (1.5,2.5) node {$4$};
\draw (1.5,3.5) node {$2$};
\draw (1.5,4.5) node {$1$};
\draw (2.5,2.5) node {$4$};
\draw (2.5,3.5) node {$2$};
\draw (2.5,4.5) node {$1$};
\draw (3.5,3.5) node {$2$};
\draw (3.5,4.5) node {$1$};
\draw (4.5,4.5) node {$1$};
\end{tikzpicture}

\begin{tikzpicture}[scale=.75]
\draw[step=1cm,black,very thin] (0,0) grid (1,5);
\draw[step=1cm,black,very thin] (1,1) grid (2,5);
\draw[step=1cm,black,very thin] (2,2) grid (3,5);
\draw[step=1cm,black,very thin] (3,3) grid (4,5);
\draw[step=1cm,black,very thin] (4,4) grid (5,5);
\draw (.5,.5) node {$1$};
\draw (.5,1.5) node {$1$};
\draw (.5,2.5) node {$1$};
\draw (.5,3.5) node {$1$};
\draw (.5,4.5) node {$1$};
\draw (1.5,1.5) node {$2$};
\draw (1.5,2.5) node {$2$};
\draw (1.5,3.5) node {$2$};
\draw (1.5,4.5) node {$2$};
\draw (2.5,2.5) node {$3$};
\draw (2.5,3.5) node {$3$};
\draw (2.5,4.5) node {$3$};
\draw (3.5,3.5) node {$4$};
\draw (3.5,4.5) node {$4$};
\draw (4.5,4.5) node {$5$};
\end{tikzpicture}
\end{multicols}
\caption{At left the cost of each cell in $B_7$, at right the space in each cell.}
\label{fig:CostAndSpace}
\end{figure}
\end{center}

 For $D$, a collection of cells, let $C(D)$ be the cost of those cells and $S(D)$ be the amount of room in those cells.

 \begin{rmk}
The space of a cell $(i,j)$ is $i$.  We chose $[n] = \{0,1,\dots, n-1\}$ for this reason.
 \end{rmk}

Our goal, finding a 3-graph on $n$ vertices having $e$ edges with the maximum number of $2$-independent sets, can be rephrased as follows:  find a downset $D$ in $B_n$ such that  $C(D)$ is minimized subject to the condition that $S(D)\geq e$.

For the rest of the paper we will only be concerned with the shape of the downset in the bottom layer.  
Given a downset in $B_n$ that has enough space to accommodate the number of edges we need we can arrange the edges in higher layers to get a shifted 3-graph (often in several ways).  When we discuss the number of $2$-independent sets in $D\subseteq B_n$ we mean the number of $2$-independent sets in any $\cH$ that has downset $D$.

Finally we introduce an order on downsets in $B_n$.  For downsets $D$ and $D'$ we say that $D$ is lex-less than $D'$, or $D<_L D'$, if 
\[\min_{\text{Lex}} D\Delta D' \in D.\]
 Here $\displaystyle{\min_{\text{Lex}} D\Delta D'}$ means the minimum cell in $D\Delta D'$ under the lex ordering on cells in $B_n$.

\begin{defi}  
A downset $D$ in $B_n$ is an \emph{optimal downset} if, for some $e$, $D$ minimizes $C(D)$ among all downsets with space at least $e$ and it is the earliest downset in lex order to do so.
\end{defi}

\section{Local Moves}
\label{sec:moves}

In this section we show certain downsets in $B_n$ do not have as many $2$-independent sets as the downset associated to a $(2,3,1)$-lex style $3$-graph.  Our strategy is to show that, given a downset $D$ that is not $(2,3,1)$-lex style, there exists a downset $D'$ such that $S(D')\geq S(D)$ and $C(D')\leq C(D)$ and $D'<_L D$.  That is, we will show that some downsets that are not $(2,3,1)$-lex style are not optimal downsets. We'll call the switch from $D$ to $D'$ a \emph{local move}.
To talk about the local moves we first need the definition of corner.

\begin{defi}
For a downset $D$ the cell $(a,b)$ is a \emph{corner} of $D$ if it is a maximal element of $D$. 
\end{defi}

The rest of this section is organized into three subsections, one for each of the three types of local moves we will perform.  In Section \ref{subsec:onecellmoves} we will perform ``one cell moves", that is, local moves in which we remove only one cell from $D$.  In Section \ref{subsec:columnmoves} we will perform ``column moves" which are local moves in which we remove a column-like subset of the downset $D$. Finally in Section \ref{subsec:largermoves} we consider a local move that removes a large subset of cells.

\subsection{One Cell Moves}
\label{subsec:onecellmoves}

 First we will consider some local moves where we exchange one cell of a downset $D$ for two cells in $B_n\setminus D$.  To do this, we first define the horizontal distance vector of a downset.
 
 \begin{defi}
 For a downset $D$, let  $(o_1,o_2,\dots,o_k)$ be the sequence of the first coordinates of the corners written in increasing order and let the \emph{horizontal distance vector} be $H(D) = (o_2-o_1,o_3-o_2,\dots, o_k-o_{k-1})$.
\end{defi}

\begin{lem}\label{lem:distancevector}
Let $D$ be a downset with horizontal distance vector $(d_1,d_2,\dots, d_k)$ where $d_i = o_{i+1}-o_i$, the difference between the first coordinates of consecutive corners.  If $3\leq d_i \leq  \frac{o_{i+1}+3}{2} $ then $D$ is not optimal.
\end{lem}
\begin{proof} Let $(a,b)$ and $(c,d)$ be consecutive corners and suppose $3\leq c-a \le \frac{c+3}{2}$.  Since the previous corner is $(a,b)$ we can remove cell $(c,d)$ and replace it with cells $(a+1,d+1)$ and $(a+2,d+1)$ and still have a downset. Let $D' = D- (c,d) + (a+1,d+1)+(a+2,d+1)$.  The move from $D$ to $D'$ is illustrated in Figure \ref{fig:ocmove}. 

\begin{figure}[!ht]
\begin{center}
\begin{tikzpicture}[scale=.75]
\draw (0,0) -- (0,6) -- (6,6) -- cycle;
\draw[fill = lightgray]  (3.5,3.5) -- (3,3.5) -- (3,4.5) -- (1.5,4.5) -- (1.5,5.5) -- (0,5.5) -- (0,0) -- cycle;
\draw [dashed] (1.5,4.75) -- (2,4.75)--(2,4.5);
\draw [dashed] (1.75,4.5) -- (1.75,4.75);
\draw[fill = black] (2.75,4.25) rectangle (3,4.5);
\draw[->] (3,4.5) -- (2.2,4.65);
\draw (-.5,4.25) node {$d$};
\draw (-.5,5.25) node {$b$};
\draw (1.4, 6.2) node {$a$};
\draw (2.9,6.2) node {$c$};
\end{tikzpicture}
\caption{Move occurring in the proof of Lemma \ref{lem:distancevector} for consecutive corners}
\label{fig:ocmove}
\end{center}
\end{figure}
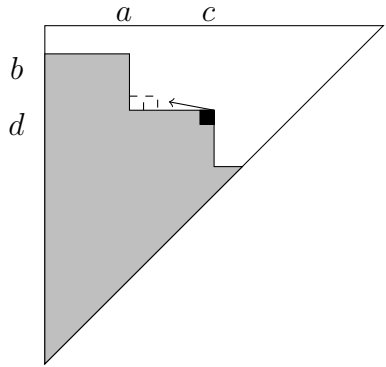

Note the room of cell $(c,d)$ is $c$ and the room in the replacement cells is collectively $2a+3$.   Since $c - a \leq \frac{c+3}{2}$ we have $c \leq 2a + 3$
and so there is at least much space in $D'$.  Moreover, the cost of each of the replacement cells is half the cost of $(c,d)$ and so $C(D) = C(D')$.  Finally $D'<_L D$.  Therefore such a $D$ is not optimal.
\end{proof}

Lemma \ref{lem:distancevector} says that in an optimal downset the horizontal distance between two corners is either small (less than 3) or is large (about half the larger amount of space).  Let's consider first when the horizontal distance between corners is small.  When the horizontal distance between two corners is 1 we will say there is a \emph{short stair} and when the horizontal distance between two consecutive corners is 2 we will say there is a \emph{long stair}.  %What types of staircases can appear in an optimal partition? 

\begin{lem}\label{lem:stairs}
Consider a downset $D$ with horizontal distance vector $H(D)$. If $H(D)$ has three consecutive 1's, two consecutive 2's, or an adjacent 1 and 2 then $D$ is not an optimal downset.
\end{lem}
\begin{proof}

In Figure \ref{fig:stairs} we show the downsets resulting from the horizontal distance vectors having three consecutive 1's, two consecutive 2's, a 1 followed by  2, and a 2 followed by  1. In each case we can show that there is a downset with at least as much space and less cost that is earlier in lex order.  

\begin{center}
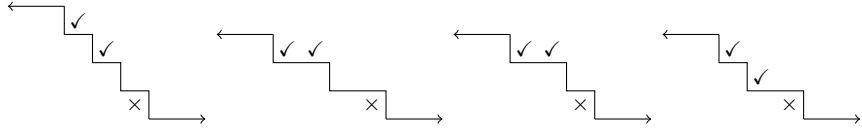
\begin{figure}[H]
\begin{tikzpicture}[scale=.75]
\draw (0,4.5)--(0,4)--(.5,4) -- (.5,3.5) -- (1,3.5) -- (1,3) -- (1.5,3) -- (1.5,2.5); 
\draw[->] (1.5,2.5) -- (2.5,2.5);
\draw[<-] (-1,4.5) -- (0,4.5);
\draw (1.25,2.75) node {\tiny{$\times$}};
\draw (.25, 4.25) node {\tiny{$\checkmark$}};
\draw (.75, 3.75) node {\tiny{$\checkmark$}};
\end{tikzpicture}
\begin{tikzpicture}[scale=.75]
\draw (0,4) -- (0,3.5) -- (1,3.5) -- (1,3) -- (2,3) -- (2,2.5);
\draw[->] (2,2.5) -- (3,2.5);
\draw[<-] (-1,4) -- (0,4);
\draw (1.75,2.75) node {\tiny{$\times$}};
\draw (.25, 3.75) node {\tiny{$\checkmark$}};
\draw (.75, 3.75) node {\tiny{$\checkmark$}};
\end{tikzpicture}
\begin{tikzpicture}[scale=.75]
\draw (0,4) -- (0,3.5) -- (1,3.5) -- (1,3) -- (1.5,3) -- (1.5,2.5);
\draw[->] (1.5,2.5) -- (2.5,2.5);
\draw[<-] (-1,4) -- (0,4);
\draw (1.25,2.75) node {\tiny{$\times$}};
\draw (.25, 3.75) node {\tiny{$\checkmark$}};
\draw (.75, 3.75) node {\tiny{$\checkmark$}};
\end{tikzpicture}
\begin{tikzpicture}[scale=.75]
\draw (0,4) -- (0,3.5) -- (.5,3.5) -- (.5,3) -- (1.5,3) -- (1.5,2.5);
\draw[->] (1.5,2.5) -- (2.5,2.5);
\draw[<-] (-1,4) -- (0,4);
\draw (1.25,2.75) node {\tiny{$\times$}};
\draw (.25, 3.75) node {\tiny{$\checkmark$}};
\draw (.75, 3.25) node {\tiny{$\checkmark$}};
\end{tikzpicture}
\caption{From left to right, 3 short stairs, 2 long stairs, 1 long stair followed by a short stair, and 1 short stair followed by a long stair. The vertical drops may be of any height at least $1$. We create downsets that are earlier in lex order by removing cells marked $\times$ and replacing them with cells marked $\checkmark$.}
\label{fig:stairs}
\end{figure}
\end{center}

Suppose that the horizontal distance vector has three consecutive 1's.
Name the corresponding corners $(i,a)$, $(i+1,b)$, $(i+2,c)$, and $(i+3,d)$ and note $a>b>c>d$.  Consider the downset $D' = D -(i+3,d) + (i+1,b+1)+(i+2,c+1)$.  Since $(i+1)+(i+2) = 2i+3> i+3$ we have $S(D')> S(D)$.  Moreover, since $a>b>c$, the cost of $(i+2,c)$ is at most half the cost of the cell $(i+3,d)$ and the cost of the cell $(i+1,b+1)$ is at most a fourth of the cost of the cell $(i+3,d)$.  Therefore $C(D')<C(D)$.

The proof for each of the other cases is similar.
\end{proof}

From Lemma \ref{lem:stairs} we know that  in an optimal downset the only possible ``staircases" are 1 long stair, 1 short stair, or 2 short stairs.  Note that these are exactly the types of staircases that appear at the end of a downset of a $(2,3,1)$-lex style hypergraph. Our next lemma describes the types of vertical drops that can appear in these transitions.

\begin{lem}\label{lem:DropSizes}
Suppose $D$ is a downset with corners $(a,b)$, $(a+1,c)$ and $(a+2,d)$.  If $b-c>1$ then $D$ is not an optimal downset.  Similarly, if $D$ is a downset with corners $(a,b)$ and $(a+2,c)$ where $b-c>1$ then $D$ is not an optimal downset.
\end{lem}
\begin{proof}
First consider a downset $D$ with corners $(a,b)$, $(a+1,c)$, and $(a+2,d)$.  If $b-c>1$ then $D' = D -(a+2,d) + (a+1,c+1) + (a+1,c+2)$ is a downset with $C(D')<C(D)$, $S(D')>S(D)$, and $D'<_L D$. For a downset $D$ with corners $(a,b)$ and $(a+2,c)$  if $b-c>1$ then $D' = D -(a+2,c) + (a+1,c+1)+(a+1,c+2)$ is a downset with $C(D')<C(D)$, $S(D')>S(D)$, and$D'<_L D$.
\end{proof}

Lemmas \ref{lem:distancevector}, \ref{lem:stairs}, and \ref{lem:DropSizes} allow us to say that optimal downsets have small groups of corners  that are ``far" apart.  The small groups (or ``transitions") look like those in Figure \ref{fig:drops} where the unlabeled drops are arbitrary.
\begin{center}
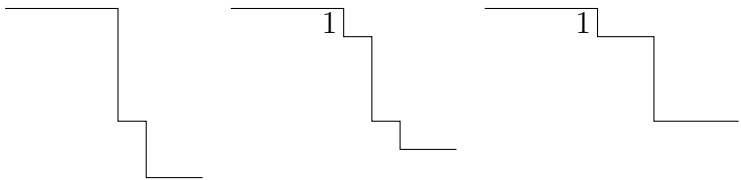
\begin{figure}[!ht]
\begin{tikzpicture}[scale=.75]
\draw (0,2) -- (2,2) -- (2,0) -- (2.5,0) -- (2.5,-1) -- (3.5,-1);
\draw (4,2) -- (6,2) -- (6,1.5) -- (6.5,1.5)  --(6.5,0) -- (7,0) -- (7,-.5) -- (8,-.5);
\draw (5.75,1.75) node {$1$};
\draw (8.5,2) -- (10.5,2) -- (10.5,1.5) -- (11.5,1.5) -- (11.5,0) -- (13,0);
\draw (10.25,1.75) node {$1$};
\end{tikzpicture}
\caption{From left to right: one short stair, two short stairs, and one long stair. The unmarked vertical drops can be of any height. These are the possible transitions in an optimal downset.}
\label{fig:drops}
\end{figure}
\end{center}

We will say that a downset \emph{ends with stairs} if the last entry of the horizontal distance vector is a 1 or a 2.  Lemmas \ref{lem:distancevector} and \ref{lem:stairs} say that if a downset ends with stairs, then it ends with 2 short stairs, 1 short stair, or 1 long stair.  In the next lemma we address downsets that end with 2 short stairs or 1 long stair and are not $(2,3,1)$-lex style.

\begin{lem}
\label{lem:EndsWithStairs}
Suppose that $D$ is not $(2,3,1)$-lex style. If $D$ ends with 2 short stairs
or 1 long stair then $D$ is not an optimal downset. 
\end{lem}

\begin{proof}
Suppose $D$ ends with 2 short stairs or 1 long stair, and there exists an earlier corner, as shown in the first two downsets in Figure \ref{fig:EndsInStairs}.  In each of these cases we can replace the last corner (marked with $\times$) with two earlier cells (marked with $\checkmark$) which cost strictly less and have at least as much space.

Suppose that $D$ ends with 2 short stairs or 1 long stair and there does not exist an earlier corner.  If the top stair $(i,j)$ has $j=n-1$ then $D$ is $(2,3,1)$-lex style.  Otherwise we can replace the last corner (marked with $\times$) with two earlier cells (marked with $\checkmark$) which have at least as much space and cost at most as much.  This results in a downset that is earlier in $(2,3,1)$-lex order. 

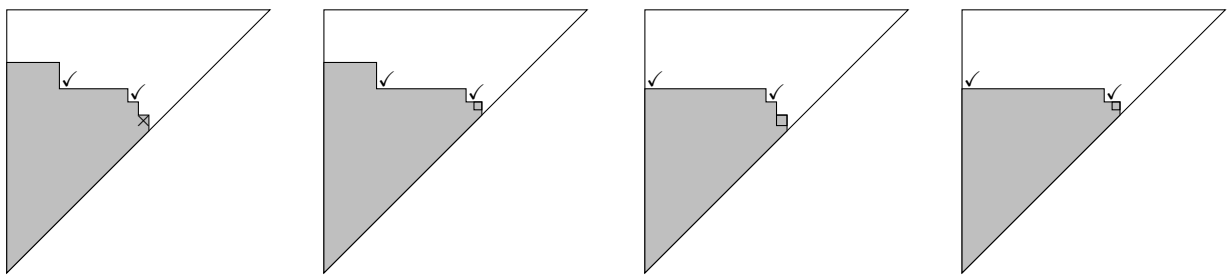
\begin{figure}[H]
\begin{multicols}{4}
\begin{tikzpicture}[scale=.7]
\draw (0,0) -- (0,5) -- (5,5) -- cycle;
\draw[fill=lightgray]  (0,4) -- (1,4) -- (1,3.5) -- (2.3,3.5) -- (2.3,3.25) -- (2.5,3.25) -- (2.5,3) -- (2.7,3) --(2.7,2.7) -- (0,0) -- cycle;
%\draw (2.5,3) -- (2.7,3) -- (2.7,2.8) -- (2.5,2.8) -- cycle;
\draw (1.2,3.7) node {{\tiny $\checkmark$}};
\draw (2.5,3.45) node {{\tiny $\checkmark$}};
\draw (2.6, 2.9) node {{\tiny $\times$}};
\end{tikzpicture}

\begin{tikzpicture}[scale=.7]
\draw (0,0) -- (0,5) -- (5,5) -- cycle;
\draw[fill=lightgray]  (0,4) -- (1,4) -- (1,3.5) -- (2.7,3.5) -- (2.7,3.25) -- (3,3.25) -- (3,3) -- (0,0) -- cycle;
\draw (2.85,3.25) -- (3,3.25) -- (3,3.1) -- (2.85,3.1) -- cycle;
\draw (1.2,3.7) node {{\tiny $\checkmark$}};
\draw (2.9,3.45) node {{\tiny $\checkmark$}};
\end{tikzpicture}

\begin{tikzpicture}[scale=.7]
\draw (0,0) -- (0,5) -- (5,5) -- cycle;
\draw[fill=lightgray]  (0,3.5) -- (2.3,3.5) -- (2.3,3.25) -- (2.5,3.25) -- (2.5,3) -- (2.7,3) --(2.7,2.7) -- (0,0) -- cycle;
\draw (2.5,3) -- (2.7,3) -- (2.7,2.8) -- (2.5,2.8) -- cycle;
\draw (0.2,3.7) node {{\tiny $\checkmark$}};
\draw (2.5,3.45) node {{\tiny $\checkmark$}};
\end{tikzpicture}

\begin{tikzpicture}[scale=.7]
\draw (0,0) -- (0,5) -- (5,5) -- cycle;
\draw[fill=lightgray]  (0,3.5) -- (2.7,3.5) -- (2.7,3.25) -- (3,3.25) -- (3,3) -- (0,0) -- cycle;
\draw (2.85,3.25) -- (3,3.25) -- (3,3.1) -- (2.85,3.1) -- cycle;
\draw (0.2,3.7) node {{\tiny $\checkmark$}};
\draw (2.9,3.45) node {{\tiny $\checkmark$}};
\end{tikzpicture}
\end{multicols}
\caption{Downsets that are not $(2,3,1)$-lex style that end in 2 short stairs or 1 long stair, with or without previous corners, are not optimal.}
\label{fig:EndsInStairs}
\end{figure}
\end{proof}

\subsection{Column Moves}
\label{subsec:columnmoves}

In this section we apply moves in which a subset of the cells in the last column of the downset are traded for a row.  These moves will be used on downsets that have that their last corner $(i,j)$ satisfies $j-i\geq \lfloor \log_2(i)\rfloor$. 
Since having a corner $(i,j)$ means the number of cells in column $j$ is $j-i$ this is ensuring that the last column of the downset has at least $\lfloor\log_2(i)\rfloor$ cells.

\begin{lem}\label{lem:ColumnMoveNoOC}
Suppose the last corner of a downset $D$ is $(i,j)$ where $j-i \geq \lfloor\log_2(i)\rfloor$, and $i\geq 5$.  If $(i,j)$ is the only corner and $j<n-1$ then $D$ is not an optimal downset. 
\end{lem}
\begin{proof}
Let $t = \lfloor\log_2(i)\rfloor$ and  define $L = \{(i,h): j-t+1 \leq h \leq j\}$ and $R =  \{(h,j+1): 1\leq h \leq i-2\}$. Consider $D' = D -L + R$. Note that we add all possible cells  in the row except for one (see Figure \ref{fig:Column1}). 
\begin{figure}[H]
\begin{center}
\begin{tikzpicture}[scale=.8]
\draw (0,0) -- (0,6) -- (6,6) -- cycle;
\draw[fill = light-gray] (0,0) -- (0,5) --  (3,5) -- (3,3) -- cycle;
\draw (2.9,6.25) node {{\scriptsize $i$}};
\draw (-.35,4.7) node {{\scriptsize $j$}};
\draw [fill = medium-gray] (3,3.5) -- (2.8,3.5) -- (2.8,5) -- (3,5) -- cycle;
\draw [decorate,decoration={brace,amplitude=6pt,mirror},xshift=-6pt,yshift=0pt] (3.2,3.5) -- (3.2,5) node [black,midway,xshift=0.8cm]{{\tiny $\lfloor\log_2(i)\rfloor$}};
\draw [dashed] (0,5.25) -- (2.5,5.25) -- (2.5,5);
\draw (2.9,5) edge[out=400,in=400,->] (1.75,5.45);
\end{tikzpicture}
\caption{Column move in the proof of Lemma \ref{lem:ColumnMoveNoOC}}
\label{fig:Column1}
\end{center}
\end{figure}
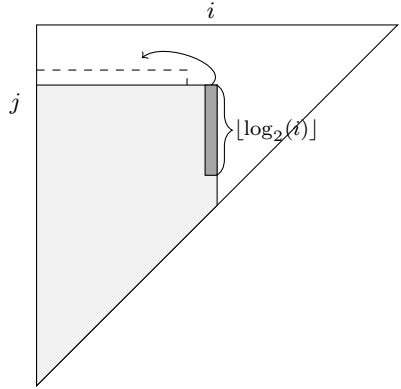

Computing the cost of $L$ and $R$ we have
\begin{align*}
C( L) &=2^{n-j-1} + 2^{n-j} + \cdots + 2^{n-j-1+t-1}\\
&=2^{n-j-1}(2^{t}-1)\\
&=2^{n-j-2}(2^{\lfloor \log_2(i)\rfloor+1}-2)
\end{align*}
and
\[C( R) = 2^{n-j-2}(i-3) + 2^{n-j-1} = 2^{n-j-2}(i-1).\]
Since $2^{\lfloor\log_2(i)\rfloor+1}-2 \geq i-1$, we have $C(D)\geq C(D')$.  Moreover,
$$S(L) = i\cdot \lfloor\log_2(i)\rfloor$$
and
$$S( R) =\frac{(i-2)(i-1)}{2}.$$
So $S(D') \geq S(D)$ when $i\geq 9$ or $i=7$.  Since $D'<_L D$ we are done if $i\geq 9$ or $i=7$.  

In the cases where $i=5,6$ or $8$ we add all possible cells in the row.  That is, we let $R =  \{(h,j+1): 1\leq h \leq i-1\}$ and leave $L$ the same.  The downsets $D' = D - L + R$ each have at most the cost of $D$, at least the space of $D$, and $D'<_L D$.
\end{proof}

\begin{lem}\label{lem:ColumnMoveEarlyOC}
Suppose a downset $D$ does not end in stairs and has last corner $(i,j)$ where $j-i\geq \lfloor\log_2(i)\rfloor$ and $i\geq 6$.  If there is an earlier corner then $D$ is not an optimal downset.
\end{lem}
\begin{proof}
Since $D$ does not end in stairs, all previous corners $(k,m)$ have $k< \frac{i-3}{2}$. Choose $(k,m)$ to be the second to last corner. Let $t = \lfloor \log_2(i) \rfloor$ and consider 
\[D' = D - \{(i,h): j-t+1 \leq h\leq j\} + \{(h,j+1):k+1\leq h\leq i-1\}.\]
  That is, we consider the downset $D'$ in which we remove $t$ cells from the last column and replace them with the available cells at height $j+1$. This move is shown in Figure \ref{fig:Column}.  

\begin{figure}[H]
\begin{center}
\begin{tikzpicture}
\draw (0,0) -- (0,6) -- (6,6) -- cycle;
\draw[fill = light-gray] (0,0) -- (0,5.5) -- (.5,5.5) -- (.5,5) -- (3,5) -- (3,3) -- cycle;
\draw (2.9,6.25) node {{\scriptsize $i$}};
\draw (-.35,4.7) node {{\scriptsize $j$}};
\draw ( .35 , 6.25) node {{\scriptsize $k$}};
\draw (-.35, 5.4) node {{\scriptsize $m$}};
\draw [fill = medium-gray] (3,3.5) -- (2.8,3.5) -- (2.8,5) -- (3,5) -- cycle;
\draw [decorate,decoration={brace,amplitude=6pt,mirror},xshift=-6pt,yshift=0pt] (3.2,3.5) -- (3.2,5) node [black,midway,xshift=0.8cm]{{\tiny $\lfloor\log_2(i)\rfloor$}};
\draw [dashed] (.5,5.25) -- (2.8,5.25) -- (2.8,5);
\draw (2.9,5) edge[out=400,in=400,->] (1.75,5.45);
\end{tikzpicture}
\caption{Column move in the proof of Lemma \ref{lem:ColumnMoveEarlyOC}}
\label{fig:Column}
\end{center}
\end{figure}
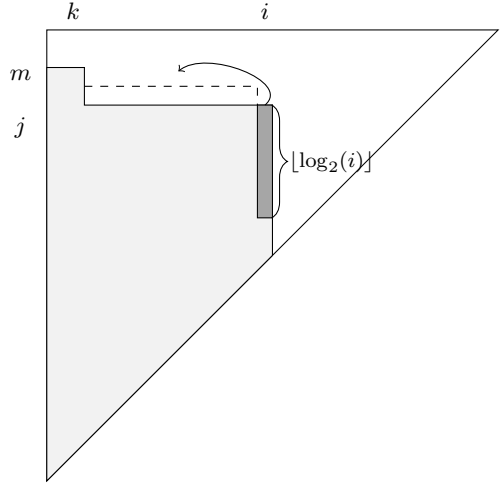

The cost of the column is
\[2^{n-j-1} + \cdots + 2^{n-j-1+t-1} = 2^{n-j-1}(2^{t}-1) = 2^{n-j-2}(2^{\lfloor \log_2(i)\rfloor + 1} -2) > 2^{n-j-2}(i-2).\]

Note there are at most $i-2$ cells in the row (since there is a previous corner) and the cost of each cell is $2^{n-j-2}$.  Thus, the cost of the row  is strictly less than the cost of the column.

The space in the column is exactly $i\lfloor \log_2(i)\rfloor$ and the space in the row is
\begin{align*}
S( \{(h,j+1):k+1\leq h\leq i-1\}) &= (k+1) + (k+2) + \cdots + (i-1)\\
&=\frac{(i-1)i}{2} - \frac{k(k+1)}{2}\\
&\geq \frac{(i-1)i}{2} - \frac{ \frac{i-4}{2}\cdot \frac{i-2}{2}}{2}\\
&= \frac{3}{8}i^2+ \frac{i}{4} -1
\end{align*}
since $k<\frac{i-3}{2}$. So $S(D')\geq S(D)$ when $i\geq 6$.  Therefore $D$ is not an optimal downset.
\end{proof}

\begin{cor}\label{cor:Tall}
Suppose that a downset $D$ does not end  in stairs, is not $(2,3,1)$-lex style, and has last corner $(i,j)$ where $j-i\geq \lfloor \log_2(i)\rfloor$ and $i\geq 5$.  Then $D$ is not an optimal downset.
\end{cor}
\begin{proof}
Consider such a downset $D$. If there is no previous corner then $j<n-1$ since $D$ is not $(2,3,1)$-lex style.  Thus, $D$ is not optimal by Lemma \ref{lem:ColumnMoveNoOC}. If $D$ has a previous corner $(i',j')$ then $i'<\frac{i-3}{2}$ by Lemma \ref{lem:distancevector}.  Then $i\geq 6$, else such a previous corner can not exist.  By Lemma \ref{lem:ColumnMoveEarlyOC} $D$ is not optimal.
\end{proof}

Corollary \ref{cor:Tall} deals with downsets that do not end in stairs and have that the last column is tall.  In the next lemmas, we will deal with downsets that end with stairs and the column of the top stair is tall.  By Lemmas \ref{lem:stairs} and \ref{lem:EndsWithStairs} we only need to consider downsets that end in one short stair.

\begin{lem}
\label{lem:TallStairs}
Suppose that the last corner of a downset $D$ is $(i',j')$ and the first corner is $(i,j)$ with $i = i'-1$.  If $j-i\geq \lfloor \log_2(i)\rfloor + 1$ and $i\geq 6$ then $D$ is not an optimal downset.
\end{lem}
\begin{proof}
 Let $t = \lfloor\log_2(i)\rfloor$.   For $h\in \{j-t+1,\ldots, j\}$ let $\ell(h)$ be the greatest integer such that $(\ell(h), h)\in D$. Note $\ell(h) \in \{i,i'\}$.  Let $L = \{(\ell(h),h): j-t+1\leq h \leq j\}$. These are cells in $B_n$ since $j-i\geq \lfloor\log_2(i) \rfloor + 1$ and $i'\leq i +1$.   Let $R = \{(m,j+1): 1\leq m \leq i-2\}$.  Consider $D' = D -L + R$
as shown in Figure \ref{fig:Column4}.
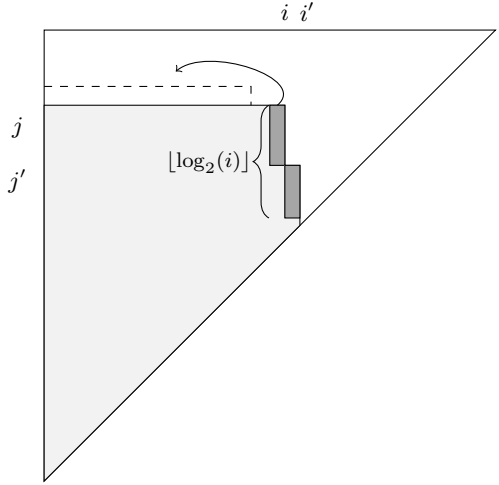
\begin{figure}[H]
\begin{center}
\begin{tikzpicture}
\draw (0,0) -- (0,6) -- (6,6) -- cycle;
\draw[fill = light-gray] (0,0) -- (0,5) --  (3,5) -- (3.2,5) -- (3.2,4.2) -- (3.4,4.2) -- (3.4,3.4) -- cycle;
\draw (3.2,6.24) node {{\scriptsize $i$}};
\draw (3.5,6.25) node {{\scriptsize $i'$}};
\draw (-.35,4.7) node {{\scriptsize $j$}};
\draw (-.35,4) node {{\scriptsize $j'$}};
%\draw [fill = medium-gray] (3,4.8) -- (2.8,4.8) -- (2.8,5) -- (3,5) -- cycle;
\draw [fill = medium-gray] (3.2,4.2) -- (3,4.2) -- (3,5) -- (3.2,5) -- cycle;
\draw [fill = medium-gray] (3.4,3.5) -- (3.2,3.5) -- (3.2,4.2) -- (3.4,4.2) -- cycle;
\draw [decorate,decoration={brace,amplitude=6pt},xshift=-6pt,yshift=0pt] (3.2,3.5) -- (3.2,5) node [black,midway,xshift=-0.8cm]{{\tiny $\lfloor\log_2(i)\rfloor$}};
\draw [dashed] (0,5.25) -- (2.75,5.25) -- (2.75,5);
\draw (3.1,5) edge[out=400,in=400,->] (1.75,5.45);
\end{tikzpicture}
\caption{Column move in the proof of Lemma \ref{lem:TallStairs}}
\label{fig:Column4}
\end{center}
\end{figure}

Since the cost of a cell (with the exception of those in the first column) only depends on the height of the cell, the cost argument is exactly the same as that of Lemma \ref{lem:ColumnMoveNoOC}.
Moreover,
$$S(L) \leq i + (i+1)( \lfloor\log_2(i)\rfloor - 1) = i\cdot \lfloor\log_2(i)\rfloor + \lfloor\log_2(i)\rfloor - 1$$
and
$$S(R) =\frac{(i-2)(i-1)}{2}.$$
Thus, $S(D') \geq S(D)$ when $i\geq 10$ or $i=7$.

In the cases where $i=6,8$ or $9$ we add all possible cells in the row.  That is, we let $R =  \{(h,j+1): 1\leq h \leq i-1\}$ and leave $L$ the same.  The downsets $D' = D - L + R$ each have at most the cost and at least the space of $D$ and are earlier in lex order. Therefore, for $i\geq 6$, such an $D$ is not an optimal downset.
%For $i=6$, note that $t=2$.  Consider removing $L$ as defined above and adding $\{(h,j+1): 1\leq h \leq 5\}$.  Then the cost of the added cells is $2^{n-j-2}(4) + 2^{n-j-1} = 3(2^{n-j-1})$ and the cost of the removed cells is $2^{n-j-1}+2^{n-j} = 3(2^{n-j-1})$.  Moreover, the space of the removed cells is at most $6+7=13$ and the space in the added cells is $15$.
%
%For  $i = 7$ consider removing $L$ as defined above and adding $\{(h,j+1): 1\leq h\leq 5\}$.  Then the cost of the added cells is $2^{n-j-2}(4) + 2^{n-j-1} = 3(2^{n-j-1})$ and the cost of the removed cells is $2^{n-j-1}+2^{n-j}=3(2^{n-j-1})$. Moreover, the space of the removed cells is at most $7+8=15$ and the space in the added cells is $15$.  
%
%For $i=8$, note that $t=3$.  Consider removing $L$ as defined above and adding $\{(h,j+1): 1\leq h \leq 7\}$.  Then cost of the added cells is $2^{n-j-2}(6) + 2^{n-j-1} = 2^{n-j+1} $ and the cost of the removed cells is $2^{n-j-1}+2^{n-j}+2^{n-j+1}$.  Moreover, the space of the removed cells is at most $8+9+9=26$ and the space in the added cells is $28$.
%
%For $i=9$, note that $t=3$.  Consider removing $L$  as defined above and adding $\{(h,j+1): 1\leq h \leq 8\}$.  Then cost of the added cells is $2^{n-j-2}(7) + 2^{n-j-1}=2^{n-j+1} + 2^{n-j-2}$ and the cost of the removed cells is $2^{n-j-1}+2^{n-j}+2^{n-j+1}$.  Moreover, the space of the removed cells is at most $9+10+10=29$ and the space in the added cells is $36$.
\end{proof}

In the next lemma consider the case where a downset ends with one short stair, there is an earlier corner, and the column of the top stair is tall.

\begin{lem}
\label{lem:TallStairs2}
Suppose that the last corner of a downset is $(i',j')$, the second to last corner is $(i,j)$ where $i = i'-1$ and there is an earlier corner with space less than $\frac{i-3}{2}$.  If $i\geq 6$ and $j - i\geq \lfloor\log_2(i)\rfloor+1$ then $D$ is not an optimal downset.
\end{lem}
\begin{proof}
 As in the proof of Lemma \ref{lem:TallStairs}, let $t = \lfloor\log_2(i)\rfloor$, for each $h\in \{j-t+1,\ldots, j\}$ let $\ell(h)$ be the greatest integer such that $(\ell(h), h)\in D$, and let $L = \{(\ell(h),h): j-t+1\leq h \leq j\}$. 
 Let $R = \{(h,j+1): 1 \leq h \leq i-1\} \cap (B_n\setminus D)$.  Consider $D' = D -L + R$, shown in Figure \ref{fig:Column5}.
\begin{figure}[H]
\begin{center}
\begin{tikzpicture}
\draw (0,0) -- (0,6) -- (6,6) -- cycle;
\draw[fill = light-gray] (0,0) -- (0,5.2) -- (1,5.2) -- (1,5)-- (3,5) --  (3.2,5) -- (3.2,4.2) -- (3.4,4.2) -- (3.4,3.4) -- cycle;
\draw (3.05,6.25) node {{\scriptsize $i$}};
\draw (-.35,4.8) node {{\scriptsize $j$}};
\draw [fill = medium-gray] (3.2,4.2) -- (3,4.2) -- (3,5) -- (3.2,5) -- cycle;
\draw [fill = medium-gray] (3.4,3.5) -- (3.2,3.5) -- (3.2,4.2) -- (3.4,4.2) -- cycle;
\draw [decorate,decoration={brace,amplitude=6pt},xshift=-6pt,yshift=0pt] (3.2,3.5) -- (3.2,5) node [black,midway,xshift=-0.8cm]{{\tiny $\lfloor\log_2(i)\rfloor$}};
\draw [dashed] (1,5.2) -- (2.9,5.2) -- (2.9,5);
\draw (3.1,5) edge[out=400,in=400,->] (1.75,5.45);
\end{tikzpicture}
\caption{Column move in the proof of Lemma \ref{lem:TallStairs2}}
\label{fig:Column5}
\end{center}
\end{figure}
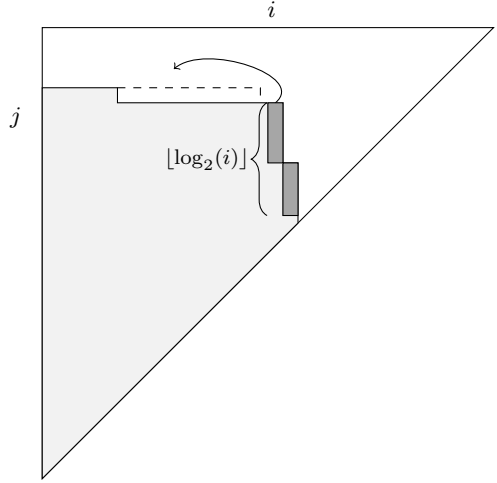

Since the cost of a cell (with the exception of those in the first column) only depends on the height of the cell, the cost argument is exactly the same as that of Lemma \ref{lem:ColumnMoveEarlyOC}.
Moreover,
$$S(L) \leq (i+1)(\lfloor\log_2(i)\rfloor -1) + i=(i+1)\lfloor\log_2(i)\rfloor -1.$$
 As in the proof of Lemma \ref{lem:ColumnMoveEarlyOC},
 $$S(R) \geq  \frac{3i^2}{8}+\frac{i}{4} - 1.$$

So $S(R) > S(L)$ for $i\geq 6$ when $i\neq 8$. If $i=8$ one can take $\floor{\log_2(i)}-1$ cells for $L$ to show $D$ is not optimal.  Since $C(D')\le C(D)$ and $S(D')\ge S(D)$ the downset $D$ is not optimal.
\end{proof}

\begin{cor}\label{cor:TallStairs}
Suppose that a downset $D$ is not $(2,3,1)$-lex style, ends in one short stair, and the top stair $(i,j)$ has $j-i\geq \lfloor \log_2(i)\rfloor + 1$ with $i\geq 6$.  Then $D$ is not an optimal downset.
\end{cor}
\begin{proof}
If $D$ has no other corners then $D$ is not optimal by Lemma \ref{lem:TallStairs}.  Now suppose $D$ ends in one short stair and has a previous corner, call it $(k,m)$.  Note $k\leq i-3$, else $D$ would end in two short stairs. If $k\geq \frac{i-3}{2}$ then $3\leq i-k\leq \frac{i+3}{2}$ and so $D$ is not optimal by Lemma \ref{lem:distancevector}.  Therefore, $k<\frac{i-3}{2}$.  By Lemma \ref{lem:TallStairs2}, $D$ is not an optimal downset.
\end{proof}

\subsection{Larger Moves}
\label{subsec:largermoves}
In this section we will consider moves that are very similar to those in the previous section.  We will trade a number of cells from the right side of a downset for the cells in the next row up.  The difference is that we allow the removed cells to come from multiple columns.  The removed cells will be those that are largest in the lex order on cells.  For two cells $(i,j)$ and $(m,k)$ we say $(i,j)\leq (m,k)$ in lex order if and only if $i<m$ or $i=m$ and $j\le k$.

\begin{lem}\label{lem:trapezoidmove2}
Suppose that $D$ is a downset that does not end in stairs with last corner $(i,j)$ such that $j-i <\lfloor \log_2(i) \rfloor$, $i\geq 16$ and $j\leq n-3$.  %Suppose that any previous corner has space at most $\frac{i-4}{2}$.  
Then $D$ is not an optimal downset.
\end{lem}

\begin{proof}

We will prove that there exists a downset $D'$ such that $C(D')\leq C(D)$, $S(D')\geq S(D)$ and that $D'<_L D$.
First we will consider the case where $2\le j-i$. Let $T$ be the $\left\lfloor\frac{i}{2}\right\rfloor$ greatest cells of $D$ under lex order.  Let $\ell$ be such that $i-\ell+1$ is the least amount of space in any cell of $T$.  That is, $T$ occupies $\ell$ columns.

Let $R$ be the cells in $B_n\setminus D$ at height $j+1$ and $j+2$ and with space at most $c =i-\ell$.  Since $j\leq n-3$, there are available cells at both height $j+1$ and $j+2$.
So 
$$R = \{(h,k): 1\leq h \leq c,j+1\leq k \leq j+2\}\cap(B_n\setminus D).$$  
The sets of cells $R$ and $T$ are shown in Figure \ref{fig:trapezoidmove1}. Let $D' = D-T+R$.

\begin{center}
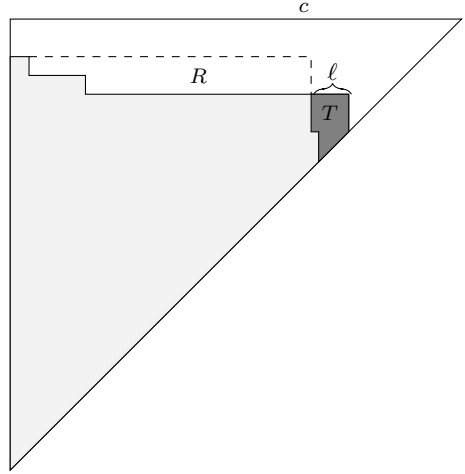
\begin{figure}[!ht]
\begin{tikzpicture}
\draw (0,0) -- (0,6) -- (6,6) -- cycle;
\draw[fill = light-gray] (0,0) --(0,5.5) -- (.25,5.5) -- (.25,5.25)-- (1,5.25) -- (1,5) -- (4.5,5) -- (4.5,4.5) -- cycle;
\draw[fill=gray] (4.5,4.5) -- (4.5,5) -- (4,5) --  (4,4.5) -- (4.1,4.5) --(4.1,4.1) -- cycle;
\draw[dashed] (.25,5.5) -- (4,5.5) -- (4,5);
\draw (3.9,6.15) node {\tiny{$c$}};
\draw (4.25,4.75) node {\tiny{$T$}};
\draw (2.5,5.25) node {\tiny{$R$}};
%\draw (.9,6.15) node {\tiny{$a$}};
\draw [decorate,decoration={brace,amplitude=4pt},xshift=1pt,yshift=0pt] (4,5) -- (4.5,5) node [black,midway,yshift=0.3cm]{\scriptsize{$\ell$}}; 
\end{tikzpicture}
\caption{A downset $D$ with $R$ and $T$ as described in the proof of Lemma \ref{lem:trapezoidmove2}.}
\label{fig:trapezoidmove1}
\end{figure}
\end{center}

First we will compare $C(T)$ and $C(R)$.  When $i\geq 18$ the size of $T$ is at least $9$.  
So the average cost of a cell in $T$ is at least $2^{n-j}$
and 
\[C(T) \geq 2^{n-j} \cdot \left\lfloor\frac{i}{2}\right\rfloor \geq 2^{n-j}\left(\frac{i}{2} - \frac{1}{2}\right) = (i-1)\cdot 2^{n-j-1}.\]
The cost of $R$ is greatest if there are no previous corners and $c=i-2$.  %Note that $\ell\geq 2$ since $\left\lfloor \frac{i}{2}\right\rfloor >\lfloor \log_2(i) \rfloor$ for all $i\geq 6$.  Thus $c\leq i-2$.    
This gives the following upper bound on $C(R)$:
$$C(R) \leq 3[(i-3)2^{n-j-3} +2^{n-j-2}] = 3(i-1)2^{n-j-3}$$
  Since $C(T) \geq 4(i-1)2^{n-j-3} \geq 3(i-1)2^{n-j-3}\geq C(R)$ we have $C(D')\leq C(D)$ when $i\geq 18$.  When $i=16$ and $i=17$ we verify by computer that a downset satisfying the constraints is not optimal.

Now we compare $S(T)$ and $S(R)$.  
Since $T$ must occupy at least $3$ columns and $|T| = \left\lfloor\frac{i}{2}\right\rfloor$, 
\begin{align*} 
S(T) &\le i\left\lfloor\frac{i}{2}\right\rfloor - \left( \left\lfloor \frac{i}{2} \right\rfloor - (\left\lfloor \lg i\right\rfloor - 1)\right)- \left( \left\lfloor \frac{i}{2} \right\rfloor - (\left\lfloor \lg i\right\rfloor)-1+\left\lfloor\lg i )\right\rfloor\right)\\
&=(i-2)\left\lfloor\frac{i}{2}\right\rfloor + 3 \left\lfloor \lg i\right\rfloor -2
\end{align*}

Note $\ell$ is greatest when $j-i$ is least.  When $j-i=2$, if $\ell =\lfloor\sqrt{i}\rfloor$ then $T$ could have up to $(\lfloor \sqrt{i}\rfloor + 1)^2/2$ cells and
$$\left(\lfloor\sqrt{i}\rfloor+ 1\right)^2 /2\geq \left(\sqrt{i}\right)^2/2  = \frac{i}{2}\geq \left\lfloor \frac{i}{2}\right\rfloor.$$ 
So, $\ell\leq \lfloor\sqrt{i}\rfloor$ and $c = i-\ell \geq i -\lfloor\sqrt{i}\rfloor$.  Let $a$ be the space in the previous corner at height $j+1$ (letting $a=0$ if there is no previous corner at height $j+1$) and let $b$ be the space in the previous corner at height $j+2$ (letting $b=0$ if there is no previous corner at height $j+2$).  Allowing both previous corners to have space $\lfloor\frac{i-4}{2}\rfloor$ gives a lower bound on $S(R)$:
\begin{align*}
S(R) &\geq [(a+1) + (a+2) + \cdots + c] + [(b+1)+(b+2) + \cdots + c]\nonumber\\
&= \frac{c(c+1)}{2} - \frac{a(a+1)}{2} + \frac{c(c+1)}{2} - \frac{b(b+1)}{2}\label{eq:space}\\
&\geq \left(i-\lfloor\sqrt{i}\rfloor\right)\left(i-\lfloor\sqrt{i}\rfloor+1\right)- \left\lfloor\frac{i-4}{2}\right\rfloor \left\lfloor\frac{i-2}{2}\right\rfloor. \nonumber
\end{align*}
So $S(R) \geq i\left\lfloor\frac{i}{2}\right\rfloor \geq S(T)$ when $i>16$.  When $i=16$, $T$ uses exactly $3$ columns and so our upper bound for $R$ can be improved and still $S(R)\geq S(T)$. Moreover, $D'<_L D$.  Therefore, $D$ is not an optimal downset in the case where $j-i\ge 2$.

When $j-i=1$ we let $T$ be the $\left\lfloor \frac{i}{2}\right\rfloor - 1$ greatest cells in lex ordering and keep $R$ the same. Via similar computations we get $S(R)\geq S(T)$, $C(T)\leq C(R)$, and $D'<_L D$ for $i\ge 16$.
\begin{comment}
Now suppose $j-i = 1$.  In this case we will let $T$ be the $\left\lfloor \frac{i}{2}\right\rfloor - 1$ greatest cells in lex ordering.  Since $j-i=1$ and $i\ge 20$ then we must use at least at least $4$ columns.  This gives us the following upper bound on the space in $T$,
	\begin{align*}
	S(T) &\leq i + 2(i-1) + 3(i-2) + \left(\left\lfloor \frac{i}{2} \right\rfloor -7\right)(i-3)\\
	&\le (i-3) \left(\left\lfloor \frac{i}{2} \right\rfloor\right) -i + 13.
	\end{align*}
Since $j-i =1$ it is sometimes the case that $T$ occupies more columns than our previous bound.  In this case, $\ell \le \lceil \sqrt{i} \rceil$ which gives us 
	\begin{align*}
	S(R) \ge (i-\lceil \sqrt{i} \rceil)(i-\lceil \sqrt{i} \rceil +1) - \left\lfloor \frac{i-4}{2} \right\rfloor \left\lfloor \frac{i-2}{2} \right\rfloor.
	\end{align*}
Then $S(R) \ge S(T)$ for $i\ge 19$.

Finally, since $i\ge 20$, $|T| \ge 9$ and the average cost of a cell in $T$ is at least $2^{n-j}$.  By the same argument as above,
	\[ C(T) \ge 2^{n-j-1}(i-3).\]
Lastly, since $\ell \ge 4$, 
	\[ C(R) \le 3(i-3) 2^{n-j-3}.\]
So $C(T) \ge C(R)$ when $i\ge 20$.  Moreover, $D'<_L D$ in this case as well. Therefore, $D$ is not optimal when $j-i=1$.
\end{comment}
\end{proof}

Lemma \ref{lem:trapezoidmove2} dealt with downsets that did not end in stairs, but the last column was short.  We now do a similar move when there is a short stair and the top stair's column is short.

\begin{lem}
\label{lem:trapezoidstairs2}
Suppose that a downset $D$ ends in one short stair and is not $(2,3,1)$-lex style.  If the last two corners $(i',j')$ and $(i,j)$ with $i = i'-1$ satisfy $j-i<\lfloor\log_2(i)\rfloor +1$, $j\leq n-3$, and $i\geq 16$ then $D$ is not an optimal downset.
\end{lem}
\begin{proof}
Again we will prove that there exists a downset $D'$ such that $C(D')\leq C(D)$, $S(D')\geq S(D)$ and that $D'<_L D$.  Let $T$ be the $\left\lfloor\frac{i}{2}\right\rfloor$ greatest cells of $D$ under the lex order.  Let $\ell$ be such that $i-\ell+1$ is the least amount of space in any cell of $T$.  So $T$ occupies $\ell+1$ columns.
The sets of cells $R$ and $T$ are shown in Figure \ref{fig:trapezoidmove2}. 
Let $D' = D-T+R$.

\begin{center}
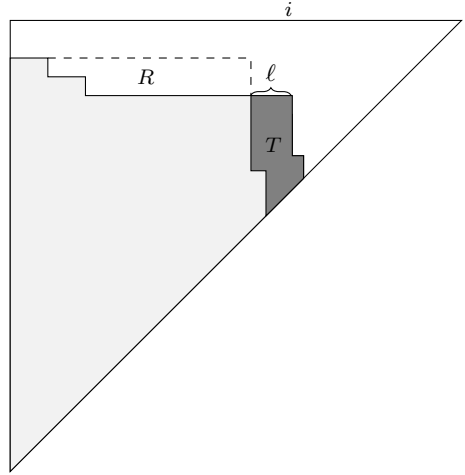
\begin{figure}[!ht]
\begin{tikzpicture}
\draw (0,0) -- (0,6) -- (6,6) -- cycle;
\draw[fill = light-gray] (0,0) -- (0,5.5) -- (.5,5.5) --(.5,5.25) --(1,5.25) --(1,5) -- (3.6,5) -- (3.6,4.75) -- (3.75,4.75) -- (3.75,4.2) -- (3.9,4.2)-- (3.9,3.9)-- cycle;
\draw[fill=gray] (3.6,5) -- (3.75,5) -- (3.75,4.2) -- (3.9,4.2)-- (3.9,3.9)-- (3.4,3.4) -- (3.4,4) -- (3.2,4) --(3.2,5)-- cycle;
\draw[dashed] (.5,5.5) -- (3.2,5.5) -- (3.2,5);
%\draw (2.9,6.15) node {\tiny{$c$}};
\draw (3.7,6.15) node {\tiny{$i$}};
%\draw (.9,6.15) node {\tiny{$a$}};
\draw (3.5,4.35) node {\tiny{$T$}};
\draw (1.8,5.25) node {\tiny{$R$}};
\draw [decorate,decoration={brace,amplitude=3pt},xshift=1pt,yshift=0pt] (3.16,5) -- (3.7,5) node [black,midway,yshift=0.3cm]{\scriptsize{$\ell$}}; 
\end{tikzpicture}
\caption{An example of a downset $D$ with $R$ and $T$ as described in the proof of Lemma \ref{lem:trapezoidstairs2}.}
\label{fig:trapezoidmove2}
\end{figure}
\end{center}
By an identical argument to Lemma \ref{lem:trapezoidmove2} we get $C(D')\leq C(D)$.
\begin{comment}
First we will compare $C(T)$ and $C(R)$. 
Since $i\geq 14$, we know $|T| \geq 7$ and so the average cost of a cell in $T$ is at least $2^{n-j}$ and 
$$C(T) \geq 2^{n-j} \cdot \left\lfloor\frac{i}{2}\right\rfloor \geq 2^{n-j}\left(\frac{i}{2} - \frac{1}{2}\right) = (i-1)\cdot 2^{n-j-1}.$$
  Moreover, $\ell\geq 2$ since $\left\lfloor \frac{i}{2}\right\rfloor > \lfloor \log_2(i) \rfloor$ for all $i\geq 6$.  So,
$$C(R) \leq 3( 2^{n-j-3}\cdot (i-3) + 2^{n-j-2} )= 3(i-1)\cdot 2^{n-j-3}.$$
  Since $C(T) \geq 4(i-1)\cdot 2^{n-j-3} \geq 3(i-1)\cdot 2^{n-j-2} \geq C(R)$ we have $C(T) \geq C(R)$ and so $C(D')\leq C(D)$.

\end{comment}
We also use a nearly identical argument to compare $S(T)$ and $S(R)$.  This time we use that the maximum space in any cell of $T$ is $i+1$ and $c\ge i- (\lfloor \sqrt{i}\rfloor-1)$ and conclude that $S(R)\ge S(T)$.
\begin{comment}
%and $|T| = \left\lfloor\frac{i}{2}\right\rfloor$, 
$$S(T) \leq \left\lfloor\frac{i}{2}\right\rfloor \cdot (i+1).$$

To get a lower bound on $S(R)$ we find a lower bound on $c$ by finding an upper bound on $\ell$. 
Note that $\ell$ is greatest when $j-i$ is least.  
Since $D$ ends in one short stair, $j-i\geq 3$.  Suppose that $j-i=3$.  
If $\ell = \lfloor\sqrt{i}\rfloor -1$ then the number of cells in the left most column of $T$ is $\lfloor\sqrt{i}\rfloor+1$ and thus the possible number of cells in $T$ is at least $\left(\lfloor\sqrt{i}\rfloor+1\right)\left(\lfloor\sqrt{i}\rfloor+2\right) -2 \geq \left\lfloor \frac{i}{2}\right\rfloor$.  
Thus $c\geq i-(\lfloor\sqrt{i}\rfloor -1)$. Let $a$ be the space in the previous corner at height $j+1$ (letting $a=0$ if there is no previous corner at height $j+1$) and let $b$ be the space in the previous corner at height $j+2$ (letting $b=0$ if there is no previous corner at height $j+2$).  Allowing both previous corners to have space $\lfloor\frac{i-4}{2}\rfloor$ gives a lower bound on $S(R)$:
\begin{align*}
S(R) &\geq [(a+1) + (a+2) + \cdots + c] + [(b+1)+(b+2) + \cdots + c]\nonumber\\
&= \frac{c(c+1)}{2} - \frac{a(a+1)}{2} + \frac{c(c+1)}{2} - \frac{b(b+1)}{2}\nonumber\\
&\geq \left(i-\lfloor\sqrt{i}\rfloor+1\right)\left(i-\lfloor\sqrt{i}\rfloor+2\right)- \left\lfloor\frac{i-4}{2}\right\rfloor \left\lfloor\frac{i-2}{2}\right\rfloor. %\label{eq:space2}
\end{align*}
So $S(R)\geq S(T)$.  
\end{comment}
Since we also have $D'<_L D$, $D$ is not an optimal downset.
\end{proof}

In the previous lemmas, we moved $\left\lfloor\frac{i}{2}\right\rfloor$ cells to two rows in the case that two rows were available.  In the next lemmas, we address if there is only one available row by moving $\left\lfloor\frac{i}{4}\right\rfloor$ cells to 1 row.

\begin{lem}
\label{lem:trapezoidmove}
Suppose that $D$ is a downset that does not end stairs with last corner $(i,j)$ such that $j-i <\lfloor \log_2(i) \rfloor$, $i\geq 23$ and $j=n-2$.  
%Suppose that any previous corner has space at most $\frac{i-4}{2}$.  
Then $D$ is not an optimal downset.
\end{lem}

\begin{proof}
We will prove that there exists a downset $D'$ such that $C(D')\leq C(D)$, $S(D')\geq S(D)$ and that $D'<_L D$.   Let $T$ be the $\left\lfloor\frac{i}{4}\right\rfloor$ greatest cells of $D$ under the lex ordering.  Let $\ell$ be such that $i-\ell+1$ is the least amount of space in any cell of $T$.  That is, $T$ occupies $\ell$ columns.

Let $R$ be the cells in $B_n\setminus D$ at height $n-1$ and with space at most $c =\min\{i-\ell,i-4\}$.  
Letting $a$ be the space in the previous corner (and $a=0$ if there is no previous corner), $R = \{(h,n-1): a+1\leq h \leq c\}$.  The sets of cells $R$ and $T$ are shown in Figure \ref{fig:trapezoidmove3}. Let $D' = D-T+R$.

\begin{center}
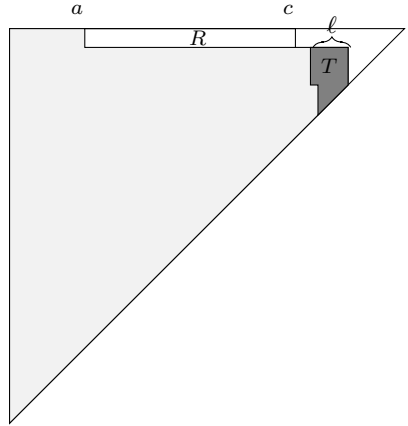
\begin{figure}[!ht]
\begin{tikzpicture}
\draw (0,0) -- (0,5.25) -- (5.25,5.25) -- cycle;
\draw[fill = light-gray] (0,0) -- (0,5.25) -- (1,5.25) -- (1,5) -- (4.5,5) -- (4.5,4.5) -- cycle;
\draw[fill=gray] (4.5,4.5) -- (4.5,5) -- (4,5) --  (4,4.5) -- (4.1,4.5) --(4.1,4.1) -- cycle;
\draw (1,5.25) -- (3.8,5.25) -- (3.8,5);
\draw (3.7,5.5) node {\tiny{$c$}};
\draw (4.25,4.75) node {\tiny{$T$}};
\draw (2.5,5.125) node {\tiny{$R$}};
\draw (.9,5.5) node {\tiny{$a$}};
\draw [decorate,decoration={brace,amplitude=4pt},xshift=1pt,yshift=0pt] (4,5) -- (4.5,5) node [black,midway,yshift=0.3cm]{\scriptsize{$\ell$}}; 
\end{tikzpicture}
\caption{A downset $D$ with $R$ and $T$ as described in the proof of Lemma \ref{lem:trapezoidmove}.}
\label{fig:trapezoidmove3}
\end{figure}
\end{center}

First we will compare $C(T)$ and $C(R)$.  When the size of $T$ is at least $9$, the average cost of a cell in $T$ is at least $4$
and 
$$C(T) \geq 4 \cdot \left\lfloor\frac{i}{4}\right\rfloor \geq 4\left(\frac{i}{4} - \frac{3}{4}\right) = (i-3).$$
  Moreover, since $c\leq i-4$,
$$C(R) \leq 1\cdot (i-5) + 2 = (i-3).$$
 When $5\leq |T|\leq |9|$ we verify by computer that a downset satisfying the constraints is not optimal. Therefore, when $i\geq 20$, $C(T) \geq C(R)$ and so $C(D')\leq C(D)$.

Now we compare $S(T)$ and $S(R)$.  Since the space in any cell of $T$ is at most $i$ and $|T| = \left\lfloor\frac{i}{4}\right\rfloor$, 
$$S(T) \leq \left\lfloor\frac{i}{4}\right\rfloor \cdot i \leq \frac{i^2}{4}.$$

Note $\ell$ is greatest when $j-i$ is least.  When $j-i=1$, if $\ell =\left\lceil\sqrt{\frac{i}{2}}\right\rceil$ then $T$ has at least 
$\left(\sqrt{\frac{i}{2}}\right)^2/2 = \frac{i}{4}\geq \left\lfloor \frac{i}{4}\right\rfloor$ cells.  
So, $i-\sqrt{\frac{i}{2}} - 1 \leq i-\ell$ and, when $i\geq 23$, $i-\sqrt{\frac{i}{2}}-1\leq i-4$.  
Thus, $c = \min\{i-4,i-\ell\} \geq i -\sqrt{\frac{i}{2}} -1$.  

Using a similar argument to that in Lemma \ref{lem:trapezoidmove2},

\[S(R)\geq \frac{3i^2}{8} - \frac{i^{3/2}}{\sqrt{2}} + \frac{i}{2} + \frac{1}{2}\sqrt{\frac{i}{2}} - 1.\]
\begin{comment}
By Lemma \ref{lem:distancevector}, any previous corner has space at most $\frac{i-4}{2}$ which gives 
\begin{align*}
S(R) &\geq (a+1) + (a+2) + \cdots + c\\
&= \frac{c(c+1)}{2} - \frac{a(a+1)}{2}\\
&\geq \frac{\left(i-\sqrt{\frac{i}{2}}-1\right)\left(i-\sqrt{\frac{i}{2}}\right)}{2}- \frac{\left(\frac{i-4}{2}\right) \left(\frac{i-2}{2}\right)}{2}\\
&=\frac{3i^2}{8}-\frac{i^{\frac{3}{2}}}{\sqrt{2}} + \frac{i}{2} + \frac{1}{2}\sqrt{\frac{i}{2}} - 1
\end{align*}
\end{comment}
So $S(R) \geq\frac{3i^2}{8}-\frac{i^{\frac{3}{2}}}{\sqrt{2}} + \frac{i}{2} + \frac{1}{2}\sqrt{\frac{i}{2}} - 1\geq \frac{i^2}{4}\geq S(T)$ when $i\geq 23$, we know that $S(D')\geq S(D)$. Since $D'<_L D$, such a $D$ is not an optimal downset.
\end{proof}

In the final lemma for this section we consider downsets similar to those of Lemma \ref{lem:trapezoidmove}, but end in one short stair.

\begin{lem}
\label{lem:trapezoidstairs}
Suppose that a downset $D$ ends in one short stair.  If the last two corners $(i',j')$ and $(i,j)$ with $i = i'-1$ satisfy $j-i<\lfloor\log_2(i)\rfloor +1$, $i\geq 23$, and $j=n-2$ then $D$ is not an optimal downset.
\end{lem}
\begin{proof} 
\begin{comment}
We will prove that there exists a downset $D'$ such that $C(D')\leq C(D)$, $S(D')\geq S(D)$ where $D'<_L D$.  Let $T$ be the $\left\lfloor\frac{i}{4}\right\rfloor$ greatest cells of $D$ under the lex ordering.  Let $\ell$ be such that $i-\ell+1$ is the least amount of space in any cell of $T$.  So $T$ occupies $\ell+1$ columns.
\end{comment}
Using the same setup as in the proof of Lemma \ref{lem:trapezoidmove}, let $R$ be the cells in $B_n\setminus D$ at height $n-1$ and with space at most $c = \min\{i-\ell, i -4\}$.  Allowing a previous corner to have space $a$ (and setting $a=0$ if there is no previous corner),  $R = \{(h,n-1): a+1\leq  h \leq c\}$.   By Lemma \ref{lem:EndsWithStairs} we know  $a\leq \frac{i-4}{2}$.  The sets of cells $R$ and $T$ are shown in Figure \ref{fig:trapezoidmove4}. Let $D' = D-T+R$.

\begin{center}
\begin{figure}[!ht]
\begin{tikzpicture}
\draw (0,0) -- (0,5.25) -- (5.25,5.25) -- cycle;
\draw[fill = light-gray] (0,0) -- (0,5.25) --(1,5.25) --(1,5) -- (3.6,5) -- (3.6,4.75) -- (3.75,4.75) -- (3.75,4.2) -- (3.9,4.2)-- (3.9,3.9)-- cycle;
\draw[fill=gray] (3.6,5) -- (3.75,5) -- (3.75,4.2) -- (3.9,4.2)-- (3.9,3.9)-- (3.4,3.4) -- (3.4,4) -- (3.2,4) --(3.2,5)-- cycle;
\draw (1,5.25) -- (3,5.25) -- (3,5);
\draw (2.9,5.5) node {\tiny{$c$}};
\draw (3.7,5.5) node {\tiny{$i$}};
\draw (.9,5.5) node {\tiny{$a$}};
\draw (3.5,4.35) node {\tiny{$T$}};
\draw (1.8,5.125) node {\tiny{$R$}};
\draw [decorate,decoration={brace,amplitude=3pt},xshift=1pt,yshift=0pt] (3.16,5) -- (3.7,5) node [black,midway,yshift=0.3cm]{\scriptsize{$\ell$}}; 
\end{tikzpicture}
\caption{An example of a downset $D$ with $R$ and $T$ as described in the proof of Lemma \ref{lem:trapezoidstairs}.}
\label{fig:trapezoidmove4}
\end{figure}
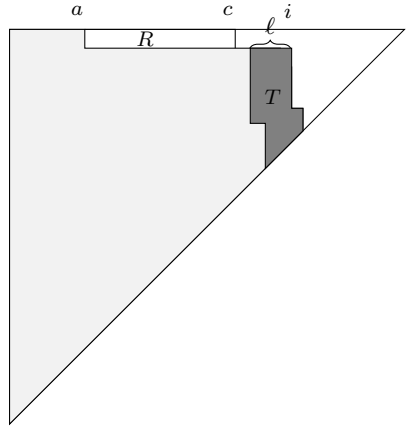
\end{center}

\begin{comment}
First we will compare $C(T)$ and $C(R)$.  Since $D$ ends in one short stair it must be the case that $j-i\geq 3$.   The average cost of a cell in $T$ would be least when $i-j=3$.  Since $i\geq 24$, then $|T| \geq 6$. So, the average cost of a cell in $T$ is at least $4$ and 
$$C(T) \geq 4 \cdot \left\lfloor\frac{i}{4}\right\rfloor \geq 4\left(\frac{i}{4} - \frac{3}{4}\right) = (i-3).$$
  Moreover, since $c\leq i-4$,
$$C(R) \leq 1\cdot (i-5) + 2= (i-3).$$
  Therefore, $C(T) \geq C(R)$ and so $C(D')\leq C(D)$.
\end{comment}

When $23\leq i<32$, we get that $c=i-4$ and by a  counting argument similar to the one in Lemma \ref{lem:trapezoidmove} we get
\[S(R) \geq \frac{(i-4)(i-3)}{2}-\frac{\left(\frac{i-4}{2}\right)\left(\frac{i-2}{2}\right)}{2}\geq \left\lfloor\frac{i}{2}\right\rfloor(i+1)\geq S(T).\]
When $i\geq 32$ by a similar argument again we get $c\geq i-\sqrt{\frac{i}{2}}$ and find $S(R)\geq S(T)$ again.

If $i\geq 23$ then the average cost of a cell is at least $4$ and by an identical argument to that of Lemma \ref{lem:trapezoidmove}, $C(D')\leq C(D)$.

 Finally, $D'<_L D$ and so such a $D$ is not an optimal downset.
\end{proof}

\section{Narrow Downsets and Persistent Exceptions}\label{sec:narrow}

Many of our lemmas thus far required that the last corner $(i,j)$ has $i\ge c$ for some small $c$.  In this section we will deal with the ``narrow" cases, that is, where $i<c$.  The first lemma deals with the case where $D$ does not end in stairs and the second lemma when $D$ ends in stairs.

There are some optimal downsets that are not $(2,3,1)$-lex style which appear as optimal downsets for all $n$.  We define
%for n\geq 28??
$\mathcal{C}_n = \{[2,1],[n-5,n-6],[n-4,n-5], [n-3,n-4,n-5],[n-3,n-4,n-5,n-6]\}$.  Let $\mathcal{H}(\mathcal{C}_n)$ be the hypergraphs generated by the partitions in $\mathcal{C}_n$.  %Let $\mathcal{L}_n$ be all the $(2,3,1)$-lex style hypergraphs on $n$ vertices.

\begin{lem}\label{lem:iSmall}
Suppose that $D$ is a downset in $B_n$ for $n\geq 10$.  Suppose $D$ does not end in stairs, $D$ is not $(2,3,1)$-lex style, $D\notin \mathcal{C}_n$, and the last corner of $D$ is $(i,j)$.  If $i<5$ then $D$ is not optimal.
 \end{lem}
\begin{proof}
Throughout this proof we use the fact that $j>i$ and that if $i<5$ then there can be no previous corners by Lemma \ref{lem:distancevector}.  If $i=1$ then $D$ is $(2,3,1)$-lex style.  If $i=2$ and $j\ge n-3$ then $D$ is $(2,3,1)$-lex style or $D\in\cC_n$.  If $i=2$ and $4\le j\le n-4$ then $D-\{(2,j),(2,j-1)\} + \{(1,j+1),(1,j+2),(1,j+3),(1,j+4)\}$ shows $D$ is not optimal.  If $j=3$ then $D\in\cC_n$.  

Suppose $i=3$.  If $j\ge n-2$ then $D\in\cC_n$ or $D$ is $(2,3,1)$-lex style.  If $5\le j\le n-3$ then $D-\{(3,j),(3,j-1)\} + \{(1,j+1),(1,j+2),(2,j+1),(2,j+2)\}$ shows that $D$ is not optimal.  If $j=4$ then, recalling $n\geq 10$, we see that $D-\{(2,4),(3,4)\} + \{(1,k): 5\le k \le 9\}$ shows $D$ is not optimal.

Finally, suppose $i=4$.  If $j\ge n-2$ then $D$ is $(2,3,1)$-style or $D\in\cC_n$.  If $6\le j\le n-3$ then $D - \{(4,j),(4,j-1)\} + \{(1,j+1),(1,j+2),(2,j+1),(2,j+2),(3,j+1),(3,j+2)\}$ shows $D$ is not optimal. If $j=5$ then $D - \{ (4,5),(3,5),(3,4)\} + \{(m,n): 1\le m \le 2, 6\le n \le 9\}$ shows $D$ is not optimal.
\end{proof}

\begin{lem}\label{lem:OneShortiSmall}
Suppose that $D$ is a downset that is not $(2,3,1)$-lex style, that $D$ ends in one short stair, and the second to last corner $(i,j)$ has $i<6$.  Then $D$ is not optimal.
\end{lem}
\begin{proof}
Let $(i,j)$ be the second to last corner and $(i',j')$ be the last corner.  By Lemma \ref{lem:distancevector} any previous corner must have space less than $\frac{i-3}{2}$.  Since $i\leq 5$ in all cases there are no previous corners.  Since $D$ is not $(2,3,1)$-lex style we know $j<n-1$.  If $i=1$ and $j=n-2$ then $D$ is $(2,3,1)$-lex style.  If $i=1$ and $j<n-2$ then there are at least two empty rows.  Then $D - (i',j') + (i,j+1)+(i,j+2)$ has the same cost and space and is earlier in lex order.  If $i=2$ then $D - (i',j') + (1,j+1)+(2,j+1)$ has the same amount of space and costs less.  If $3\le i \le 5$ then $D -(i',j')+(1,j+1)+(2,j+1)+(3,j+1)$ has at least as much space and at most the cost and is earlier in lex order.  %(Cost argument: $(1,j+1)$ has at most $1/2$ the cost and each of $(2,j+1)$ and $(3,j+1)$ has at most $1/4$ the cost.)
\end{proof}

\section{Downset Extensions}\label{sec:Extensions}

In this section we will consider a downset $D$ in $B_n$ inside $B_{\ell}$ for $\ell>n$.  We will show that if $D$ is not optimal in $B_n$ then $D$ is not optimal in $B_{\ell}$ either, and a similar lemma for when $D$ is optimal.
 \begin{defi}
  Given a downset $D$ in $B_n$, let the \emph{extension} of $D$, denoted $\overline{D}$, be the downset in $B_{n+1}$ where $(i,j)\in\overline{D}$ if and only if $(i,j)\in D$.
  \end{defi}
  
  \begin{lem}\label{lem:extension}
  Suppose $D$ is not an optimal downset in $B_n$.  Then the extension of $D$ is not optimal in $B_{n+1}$.
  \end{lem}
  \begin{proof}
  Suppose that $D$ is not an optimal downset in $B_n$ and let $\overline{D}$ be the extension of $D$.    Then there exists a downset $D'$ in $B_n$ such that $S(D')\geq S(D)$, $C(D')\leq C(D)$, and $D'$ is earlier in lex order.  We claim that $\overline{D}$ is not an optimal downset in $B_{n+1}$.  Consider $\overline{D'}$.  Then 
 
  $$S(\overline{D'}) = \sum_{(i,j)\in \overline{D'}} i = \sum_{(i,j)\in D'} i = S(D') \geq S(D) = \sum_{(i,j) \in A} i = \sum_{(i,j)\in\overline{D}} i = S(\overline{D}).$$ 
  
Recall the cost of a cell $(i,j) \in D$ with $i\neq 1$ is $2^{n-j-1}$.  The cost of the same cell $(i,j)$ in $\overline{D}$ is $2^{(n+1)-j-1} = 2(2^{n-j-1})$.  This works similarly when $i=1$ and thus the cost of $D$ is half the cost of its extension. %($C(D) = \frac{1}{2} C(\overline{D})$).  
So,
  $$C(\overline{D'}) =  2 C(D')\geq 2 C(D) = C(\overline{D}).$$
  
  Moreover, $\overline{D'}<_L\overline{D}$ since $D'<_LD$ and our definition for the lex ordering on downsets is independent of $n$. Therefore, if $D$ is not an optimal downset in $B_n$ then its extension is not an optimal downset in $B_{n+1}$.
    \end{proof}

\begin{lem}\label{lem:4EmptyRows}
Let $D$ be a downset in $B_n$ with first corner $(a,b)$ where $n-1-b\geq 4$ and last corner $(i,j)$ where $i\geq 6$.  Then $D$ is not optimal.
\end{lem}
\begin{proof}
Let $D$ be such a downset.  We will construct a downset $D'$ that has at least as much space and costs at most as much.    Let $T$ be the $\left\lfloor\frac{i}{2}\right\rfloor$ greatest cells of $D$ under lex ordering.  If $(c,d)$ is the top cell in column $c$, let $S(c) = \{(c,d+1), (c,d+2), (c,d+3), (c,d+4)\}$ and let 
$$\mathcal{S} = \bigcup_{c=1}^{\lfloor\frac{i}{2}\rfloor} S(c).$$

Let $D'= D -T + \mathcal{S}$.  First we claim that $C(D') \leq C(D)$.  If $(c,d)$ is the top cell in column $c$ there is a corresponding cell $(e,f)$ in $T$ such that $f\leq d$ and so, if $c\neq 1$, then $C(S(c))=2^{n-d-1}-2^{n-d-5} \leq 2^{n-f-1}=C((e,f))$.  This argument holds for each column of $\mathcal{S}$ with a distinct cell of $T$. The cost of the first column is double, but there is at least one cell with height at most $b-1$ in $T$ that accounts for this.

Next we claim $S(D')\geq S(D)$.  In the adding of 4 rows we get that the new space is
$$4 \cdot \frac{(\lfloor i/2 \rfloor)(\lfloor i/2\rfloor + 1)}{2}.$$
The space in the removed cells is at most $i\cdot \left(\left\lfloor \frac{i}{2}\right\rfloor\right)$ so $S(D')\geq S(D)$.  Finally, $D'<_L D$.
\end{proof}

\begin{cor}\label{cor:ExtensionOfOptimals}
For $n\ge 10$, if $D$ is optimal in $B_n$ and $D\neq [2,1]$ then $D$ is not optimal in $B_{n+4}$.
\end{cor}
\begin{proof}
Suppose $D$ is optimal in $B_n$, the first corner of $D$ is $(a,b)$ and the last corner of $D$ is $(i,j)$.  Note that $(n+4)-1-b \geq 4$ and so if $i\geq 6$ then $D$ is not optimal in $B_{n+4}$ by Lemma \ref{lem:4EmptyRows}.  Suppose now that $i\leq 5$.  Then $D$ is not $(2,3,1)$-style in $B_{n+4}$ and $D\notin \mathcal{C}_{n+4}$ since then $D$ would not fit inside $B_n$.  Thus, by Lemma \ref{lem:iSmall} $D$ is not optimal in $B_{n+4}$.
\end{proof}

\section{Upper Bounds on $n$}\label{sec:UBonN}

\begin{lem}\label{lem:FitsInside1}
Suppose that $D$ does not end in stairs and has last corner $(i,j)$ with $j-i<\floor{\lg i}$, $j=n-2$. %, and $i<36$.  
Then $D$ is not optimal for any $n\geq 32$.
\end{lem}

\begin{proof}
By Lemma \ref{lem:trapezoidmove} we know such a $D$ is not optimal when $i\geq 23$.  Since $j-i<\floor{\lg i}$ and $i<23$ then $j-i \leq 4$.  Thus, $i\leq 22$ and $j\leq 26$.  Since $j=n-2$ we know $n\leq 28$. If $D$ is not optimal in $B_{28}$ then $D$ is not optimal in $B_n$ for any $n\geq 28$.  If $D$ is optimal in $B_{28}$ then $D$ is not optimal in $B_n$ for $n\geq 32$.
\end{proof}

\begin{comment}
\begin{proof}
Since $j-i<\floor{\lg i}$ and $i<36$ we know that $j-i \leq 5$.  Thus, $i\leq 35$ and $j\leq 40$.  Since $j=n-2$ we know $n\leq 42$. If $D$ is not optimal in $B_{42}$ then $D$ is not optimal in $B_n$ for any $n\geq 42$ by Lemma \ref{lem:extension}.  If $D$ is optimal in $B_{42}$ then by Corollary \ref{cor:ExtensionOfOptimals} $D$ is not optimal in $B_n$ for $n>45$.
\end{proof}
\end{comment}

\begin{lem}\label{lem:FitsInside3}
Suppose that $D$ does not end in stairs and has last corner $(i,j)$ with $j\leq n-3$ and $j-i <\floor{\lg i}$. Then $D$ is not optimal for any $n\geq 30$ or $D=[2,1]$.
\end{lem}
\begin{proof}
We know such a $D$ is not optimal for $i\geq 16$.  Thus, for any optimal $D$ $i\leq 15$ so $j-i\leq 3$ and $j\leq 18$. If there is no previous corner then $D$ fits inside $B_{19}$ so isn't optimal for $n\geq 23$.  If there is a previous corner then we can replace $(i,j)$ with a cell at every level not in the first column so a previous corner $(i',j')$ has to have $j'<j+i/2\leq 18+(15/2)=25.5$. So $D$ fits inside $B_{26}$ and is not optimal for $n\geq 30$.

\begin{comment}
Since $i<20$ and $j-i<\floor{\lg i}$ we know that $j-i\leq 4$. Thus, $i\leq 19$ and $j\leq 23$.  If there is no previous corner then $D$ fits inside $B_{27}$ and hence $D = [2,1]$ or is not optimal for $n\geq 31$ by Lemma \ref{lem:extension} and Corollary \ref{cor:ExtensionOfOptimals}.  If there is a previous corner, note that we can replace $(i,j)$ with one cell at each height and still save on cost.  Thus, any previous corner $(i',j')$ must have that $j'<j+\frac{i}{2}$ so $j'\leq 32$.  So any optimal $D$ fits inside $B_{33}$.  Thus, by Lemma \ref{lem:extension} and Corollary \ref{cor:ExtensionOfOptimals}, $D$ is only optimal if $n\leq 36$.
\end{comment}
\end{proof}

\begin{lem}\label{lem:StairsN-2}
Suppose $D$ is a downset that ends with one short stair with the top stair being $(i,j)$.  Additionally assume $j-i<\floor{\lg i}+1$, $j=n-2$.  If $n\geq 32$ then $D$ is not optimal.
\end{lem}
\begin{proof}
By Lemma \ref{lem:trapezoidstairs} we know such a $D$ is not optimal if $i\ge 23$.  Suppose $i\le 22$.  Since $j-i<\left\lfloor \lg i \right\rfloor + 1$, $j- i \leq 4$.  So $j\leq 26$ and since $j=n-2$, $n\leq 28$.  If $D$ is not optimal in $B_{28}$ then $D$ is not optimal in $B_n$ for any $n \ge 28$. If $D$ is optimal in $B_{28}$ then $D$ is not optimal in $B_n$ for $n\geq 32$.
\end{proof}

\begin{lem}\label{lem:StairsN-3}
Suppose $D$ is a downset that ends with one short stair with the top stair being $(i,j)$.  Additionally, assume $j-i<\floor{\lg i} + 1$ and $j\leq n-3$.  If $n\geq 30$ then $D$ is not optimal.
\end{lem}
\begin{proof}
By Lemma \ref{lem:trapezoidstairs2} we know such a downset is not optimal if $i\geq 16$. Since $j < \floor{\lg i} + i + 1$ and $i\le15$ we know  $j\le 18$.  If there is no previous corner then $D$ fits inside $B_{19}$.  Suppose there is a previous corner.  Note that we can replace $(i,j)$ with one cell at each height and still save on cost.  Since $i<16$ if we have a previous corner $(i',j')$ with $j'\geq j+8$ then there is a downset $D$ with lower cost, at least as much space, and is earlier in lex order.  Thus, $j'< 18+8 = 26$ and $D$ fits inside $B_{26}$. By Lemmas \ref{lem:extension} and \ref{cor:ExtensionOfOptimals} we know $D$ is not optimal if $n\geq 30$. 
\end{proof}

\section{Proof of Theorem}
\label{sec:proof}
 \begin{prop}\label{prop:stairs}
Suppose $\cH \in \cH(n,e)$ is not $(2,3,1)$-lex style.  Let $D = D(\cH)$ and suppose $D$ ends in stairs.  If $\cH$ is optimal then $n< 32$.
\end{prop}
\begin{proof}
The cases for this proof are outlined in Figure \ref{fig:stairs}. By Lemma \ref{lem:stairs}, $D$ ends in one short stair, two short stairs, or one long stair.  By Lemma \ref{lem:EndsWithStairs}, no optimal $D$ ends in two short stairs or one long stair.  

Suppose that $D$ ends in one short stair and let $(i,j)$ be the second to last corner.  If $i<6$ then $D$ is not optimal by Lemma \ref{lem:OneShortiSmall}.  Suppose that $i\geq 6$.  If $j-i\ge \floor{\lg i} + 1$ then $D$ is not optimal by Corollary \ref{cor:TallStairs}.  

So, suppose $j-i<\floor{\log_2(i)} + 1$.  If $j=n-1$ then $D$ is $(2,3,1)$-lex style.  If $j=n-2$  then $n< 32$ by Lemma \ref{lem:StairsN-2}.   Finally,  if $j\le n-3$ then $n<30$ by Lemma \ref{lem:StairsN-3}.  Therefore, if such an $\cH$ is optimal then $n<32$.
\end{proof}

%\vspace{1in}
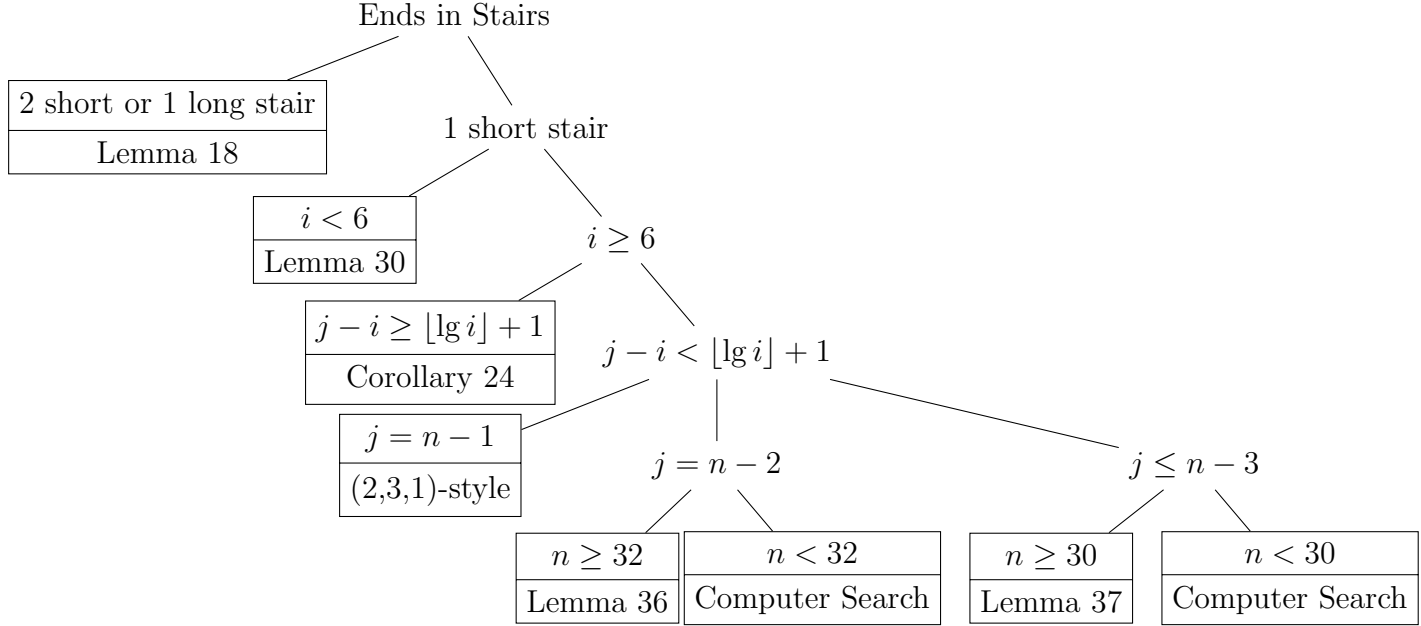
\begin{figure}
	\begin{center}
		\begin{tikzpicture}[
				level 1/.style={sibling distance=3in},
				level 2/.style={sibling distance=2in},
				level 3/.style={sibling distance=1in},
				level 4/.style={sibling distance=0.5in},
				level 5/.style={sibling distance=1.25in}
			]
			\node {Ends in Stairs}
			[
			 leaf/.style={rectangle split, rectangle split parts=2, rectangle split horizontal=false,draw}
			] % Applies to all children
			child {
				node[leaf] {$2$ short or $1$ long stair\nodepart{two}Lemma \ref{lem:EndsWithStairs}}
			}
			child[sibling distance=.75in] {
				node {$1$ short stair}
				child { node[leaf] {$i<6$\nodepart{two}Lemma \ref{lem:OneShortiSmall} }
					  }
				child[sibling distance=1in] { node {$i\ge 6$}
						child[sibling distance=2in] { node[leaf] {$j-i\ge \floor{\lg i} + 1$\nodepart{two}Corollary \ref{cor:TallStairs} } }
						child { node {$j-i<\floor{\lg i} + 1$} 
								child[sibling distance=1.5in] { node[leaf] {$j=n-1$\nodepart{two}(2,3,1)-style} }
								child[sibling distance = 2in] { node {$j=n-2$} 
									child { node[leaf] {$n\geq 32$\nodepart{two}Lemma \ref{lem:StairsN-2}}}
									child[sibling distance=1in]{ node[leaf] {$n<32$\nodepart{two}Computer Search}
								}
								}
								child[sibling distance = 2.5in] {node {$j\le n-3$} 
									child[sibling distance = 1.5in] {node[leaf] {$n\geq 30$\nodepart{two}Lemma \ref{lem:StairsN-3}}}
									child[sibling distance = 1in] {node[leaf] {$n<30$\nodepart{two}Computer Search}}
									}}
							  }
			};
		\end{tikzpicture}
	\end{center}
	\caption{The cases in the proof of Proposition \ref{prop:stairs}. Here $(i,j)$ is the top stair.}
	\end{figure}

 \begin{prop}\label{prop:nostairs}
Suppose $\cH \in \cH(n,e)$ is not $(2,3,1)$-lex style.  Let $D = D(\cH)$ and suppose $D$ does not end in stairs.  If $\cH$ is optimal then $D \in \mathcal{P}_n$ or $n<32$.
\end{prop}
\begin{proof}
Let $(i,j)$ be the last corner of $D$.  Suppose that $j-i\ge\floor{\lg i}$.  When $i\geq 5$, Corollary \ref{cor:Tall} tells us that $D$ is not optimal.  If $i<5$ then, by Lemma \ref{lem:iSmall}, $D\in \mathcal{P}_n$, $D$ is not optimal, or $n<10$.

Now suppose that $j-i<\floor{\lg i}$. If $j = n-1$ then $D$ must be $(2,3,1)$-lex style.  If $j=n-2$ then by Lemma \ref{lem:FitsInside1}, $n<32$. Finally, if $j\leq n-3$ then $n\le 30$ or $D\in\mathcal{P}_n$ by Lemma \ref{lem:FitsInside3}.
\end{proof}

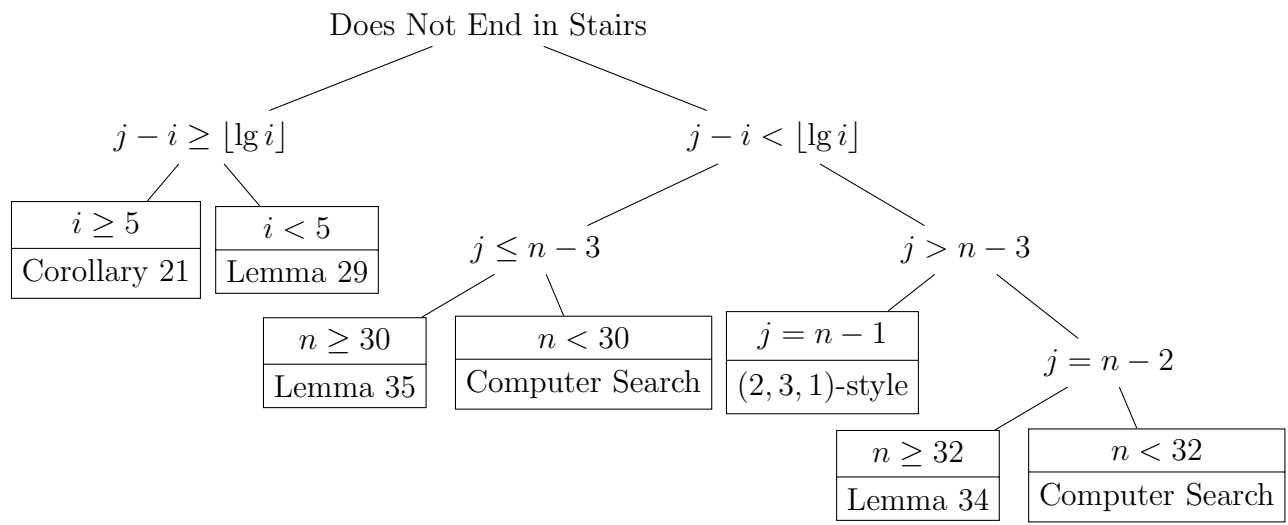
\begin{figure}	
\begin{center}
		\begin{tikzpicture}[
				level 1/.style={sibling distance=3in},
				level 2/.style={sibling distance=2in},
				level 3/.style={sibling distance=.5in},
				level 4/.style={sibling distance=0.5in},
			]
			\node {Does Not End in Stairs}
			[
			 leaf/.style={rectangle split, rectangle split parts=2, rectangle split horizontal=false,draw}
			] % Applies to all children
			child {
				node {$j-i  \ge \floor{\lg i}$}
				child[sibling distance=1in] { node[leaf] {$i\ge 5$\nodepart{two}Corollary \ref{cor:Tall}} }
				child[sibling distance=1in] { node[leaf] {$i<5$\nodepart{two}Lemma \ref{lem:iSmall}} }											   
					  						}
			child[sibling distance=3in] {
				node {$j-i<\floor{\lg i}$}
				child[sibling distance = 2.5in] { node {$j\le n-3$}
							child[sibling distance = 2in]{ node[leaf] {$n\geq 30$\nodepart{two} Lemma \ref{lem:FitsInside3}}}
							child{ node[leaf] {$n<30$\nodepart{two} Computer Search}}
						}
					%  }
				child[sibling distance=2in] { node {$j>n-3$}
						child[sibling distance = 1.5in] { node[leaf] {$j=n-1$\nodepart{two}$(2,3,1)$-style} }
						child[sibling distance = 1.5in] { node {$j=n-2$} 
								child[sibling distance=2in] { node[leaf] {$n\geq 32$\nodepart{two}Lemma \ref{lem:FitsInside1}} }
								child { node[leaf] {$n<32$\nodepart{two} Computer Search } }
							  }
					  }
			};
		\end{tikzpicture}
	\end{center}
\caption{The cases in the proof of Proposition \ref{prop:nostairs}. Here $(i,j)$ is the last corner.}
\end{figure}

%There are exactly 5 persistent counterexamples ($3$-graphs that appear for every $n$ as maximizers of $i_{2}$, but are not $(2,3,1)$-lex style). They have the following shadow graphs:
%\begin{multicols}{2}
%	\begin{itemize}
%	\item $(K_3\vee E_1) \cup E_{n-4} = K_4$
%	\item $(K_2\vee E_{n-5})\cup E_3$
%	\item $(K_2\vee E_{n-4})\cup E_2$
%	\item $(K_3\vee E_{n-4})\cup E_1$
%	\item $(K_4\vee E_{n-5}) \cup E_1$
%	\end{itemize}
%\end{multicols}

\noindent\emph{Proof of Theorem \ref{thm:2indLARGE}.} When $n\geq 32$ we find the only optimal $3$-graphs are $(2,3,1)$-lex style or are in $\mathcal{P}_n$ by Propositions \ref{prop:stairs} and \ref{prop:nostairs}. When $n<32$ we find all hypergraphs that maximize $2$-independent sets using a computer search which leads us to $(2,3,1)$-lex style graphs or those with shadow graphs shown in Table \ref{table:exceptions}.\hfill $\square$

\section{Conclusion}

We have found the maximum number of $s$-independent sets in $n$ vertex $3$-uniform hypergraphs with $e$ edges for all possible $n,e$ and $s$. While the answer is straightforward for $s=1$ and $s=3$, the answer for $s=2$  requires a generalization of lex and colex graphs to $\pi$-lex graphs. Sadly the result is not as straightforward as saying that the optimal hypergraphs are $(2,3,1)$-lex initial segments. Even the generalization to $(2,3,1)$-lex style doesn't cover all the cases. There are both transient and persistent exceptions that are not $(2,3,1)$-lex style. 

It still seems to us possible that asymptotically we can give a good characterization of the $r$-graph on $n$ vertices having $e$-edges having the fewest $s$-independent sets. The following conjecture is a strengthened version of the main theorem (Theorem 5) of \cite{HypergraphIndSets}.

\begin{conj}
    Fix $1\le s\le r$ and $\eta>0$. Let $\cH$ be a hypergraph on $n$ vertices with $e$ edges (where $e < (1-\eta) \binom{n}r$) having the maximum number of $s$-independent sets. Let $\cP(e)$ be the initial segment of $\binom{[n]}r$ in the $(r-s+1,r-s+1,\dots,r,1,2,\dots,s)$-lex order. Then
    \[
        i_s(\cH) \le (1+o(1)) i_s(\cP(e)).
    \]
\end{conj}

The case $r=3, s=2$ is a consequence of our main theorem. For all $r$ the cases $s=r$ and $s=1$ are also proved. The case $s=r$ is a special case of Theorem~\ref{thm:Sind}. The case $s=1$ is true because the argument of Lemma~\ref{lem:IsolatesAre1ind} applies equally well to the initial segments in $(r,1,2,\dots,r-1)$-lex.

One open problem to consider, which is probably very hard, is the level sets problem. For instance, one could try to determine which $3$-uniform hypergraph with $n$ vertices and $m$ edges maximizes the number of $2$-independent sets of size $t$. Our result doesn't answer this question. As an example, we know that for $12$ vertices and $10$ edges, the $(2,3,1)$-lex hypergraph maximizes the number of $2$-independent sets in total, however, this graph does not maximize the $2$-independent sets of size $2$ (at the very least the colex graph does better).

% For any $r$, maximizing $1$-independent sets and $r$-independent sets in $r$-uniform hypergraphs is straightforward. The ``middle" cases are more interesting. As shown in the introduction, the case of $r=3$ and $s=2$, is equivalent to maximizing independent sets in an $n$ vertex graph that has at least $m$ triangles. We show this by considering the downset of any hypergraph $\cH$. Phrased differently, we are minimizing the size of the total(?) upshadow of the downshadow of $\cH$. As a similar Kruskal-Katona type question, one could ask, which family of $r$-sets has the smallest $i^{th}$ upshadow of the $j^{th}$ downshadow. Not sure how to phrase this....

\bibliographystyle{amsplain}
\bibliography{Max2Indsets}

\end{document}